\documentclass[12pt]{article}
\usepackage{amsmath,graphicx,amsthm,amsfonts}
\usepackage{mathtools,amssymb,bbm,cite,bm,natbib}

\theoremstyle{definition}

\theoremstyle{plain}
\newtheorem{proposition}{Proposition}
\newtheorem{theorem}{Theorem}
\newtheorem{lemma}{Lemma}
\newtheorem{corollary}{Corollary}
\theoremstyle{remark}

\newcommand{\Rnk}[1]{\ensuremath{\operatorname{rk}\!\left(#1\right)}}
\newcommand{\Img}[1]{\ensuremath{\operatorname{Im}\!\left(#1\right)}}
\newcommand{\Ker}[1]{\ensuremath{\operatorname{Ker}\!\left(#1\right)}}
\newcommand{\Nul}[1]{\ensuremath{\operatorname{nul}\!\left(#1\right)}}
\newcommand{\Asy}[2]{\ensuremath{\underset{#1 \to #2}{\sim}}}
\newcommand{\Pro}[1]{\ensuremath{\mathbbm{P}\!\left(#1\right)}}
\newcommand{\Tmat}[1]{\ensuremath{\matr{T}(#1)}}
\newcommand{\Smat}[1]{\ensuremath{\matr{\Sigma}(#1)}}
\DeclareMathOperator{\lcm}{lcm}
\newcommand{\vect}[1]{\bm{#1}}
\newcommand{\matr}[1]{\bm{#1}}

\begin{document}
%
%

\noindent DISCRETE CONVOLUTION STATISTIC FOR HYPOTHESIS TESTING
\vskip 3mm

\vskip 5mm
\noindent Giulio Prevedello, p.giulio@hotmail.it 

\noindent Ken R. Duffy, ken.duffy@mu.ie

\noindent Hamilton Institute

\noindent Maynooth University

\noindent Maynooth, Ireland 

\vskip 3mm

\noindent ABSTRACT

The question of testing for equality in distribution between two
linear models, each consisting of sums of distinct discrete independent
random variables with unequal numbers of observations, has emerged
from the biological research. In this case, the computation of
classical $\chi^2$ statistics, which would not include all observations,
results in loss of power, especially when sample sizes are small.
Here, as an alternative that uses all data, the nonparametric maximum
likelihood estimator for the distribution of sum of discrete and
independent random variables, which we call the convolution statistic,
is proposed and its limiting normal covariance matrix determined.
To challenge null hypotheses about the distribution of this sum,
the generalized Wald's method is applied to define a testing statistic
whose distribution is asymptotic to a $\chi^2$ with as many degrees of
freedom as the rank of such covariance matrix. Rank analysis also
reveals a connection with the roots of the probability generating
functions associated to the addend variables of the linear models.
A simulation study is performed to compare the convolution test
with Pearson's $\chi^2$, and to provide usage guidelines.
\vskip 4mm

\noindent 1.   INTRODUCTION

In this paper we examine the problem of testing the null hypothesis
of equality in distribution, denoted $\sim$, for two linear models
with distinct observables, that is
\begin{equation}\label{pb:intro}
\mathbf{H}_0\colon a_0 + \sum_{i=1}^{k} a_i A_i \sim b_0 + \sum_{i=1}^{h} b_i B_i.
\end{equation}
We assume that the random variables $A_1, \ldots, A_k, B_1, \ldots,
B_h$ are bounded, independent, of possibly different distribution
and all take values in a real lattice $\Lambda(\zeta) = \{\zeta u
\colon u \in \mathbb{Z}\}$ for some $\zeta \in \mathbb{R}$ and that
$a_0, b_0 \in \Lambda(\zeta)$, $a_1, \ldots, a_k, b_1, \ldots, b_h
\in \mathbb{Z}$, the set of integers.

Equality in distribution between random variables can be tested
using statistics such as Pearson's $\chi^2$ \citep{Pearson1900} or
the more general power-divergence family \citep{CressieRead1984}.
The computation of these statistics, however, assumes that the
number of observations of each of $A_1, \ldots, A_k$ (and $B_1,
\ldots, B_h$) are equal. Otherwise, it would
seem that the data sets must be truncated for application of those
methods, which could prove wasteful if samples come in unequal
counts and their collection is costly or laborious.

For example, consider a problem in meta-analysis, where two studies
are described by linear models with distinct independent variables,
and we wish to test for equality in distribution between these
models as in \eqref{pb:intro}.  In the simplest case, for $k=2$,
$h=1$ with $\mathbf{H}_0\colon A_1+A_2 \sim B_1$, the independent
variables observed are $n_1$ distributed as $A_1$, $n_2$ as $A_2$
and $n_3$ as $B_1$, respectively noted $\{A_{11}, \ldots, A_{1n_1}\}$,
$\{A_{21}, \ldots, A_{2n_1}\}$ and $\{B_{11}, \ldots, B_{1n_3}\}$,
with $n_1, n_2, n_3 \in \mathbb{N}$.  This scenario may arise because
the independent variables are grouped differently in the studies
(e.g., $A_1$ occurrences of event $E_1$, $A_2$ occurrences of event
$E_2$, $B_1$ occurrences of any event $E_1$ or $E_2$) or because
the model choice is different (e.g., model one is $A_1 + A_2$ and
model two is $B_1 \sim f(A_1, A_2)$ for a given function $f$).
Then, to test $\mathbf{H}_0$, Pearson's $\chi^2$ could be computed
using $B_1, \ldots, B_{n_3}$ and the data from $A_1$ and $A_2$
paired as, for example, $\{A_{11}+A_{21}, \ldots, A_{1m}+A_{2m}\}$
with $m=n_1=n_2$, so that these $m$ variables are independent and
identically distributed as $A_1+A_2$ to comply with Pearson's
statistic assumptions. If observation sizes are unequal, e.g.
$n_1>n_2$, then $n_1-n_2 > 0$ variables from $\{A_{11}, \ldots,
A_{1n_1}\}$ could be excluded from the calculation of $\{A_{11}+A_{21},
\ldots, A_{1m}+A_{2m}\}$, now with $m=\min(n_1, n_2)$. But any
pairing or variables exclusion are two choices that, either arbitrarily
or randomly determined, influences the outcome of the test.

We were personally motivated to address this question during a study
on how stimulatory signals are integrated by the immune system
\citep{Marchingo2016}. We wished to statistically test the
hypothesis that the expansion impetus of two stimulii were integrated
independently by cells when the signals were provided together,
which had been hypothesized in a previously published study
\citep{Marchingo2014}. The experimental data obtained for that work
was costly to produce, both in terms of manpower and reagents, and
inherently came with distinct numbers of observations of all
variables. Thus we sought to develop a statistical procedure that utilized
all available data. The resulting test may prove useful in other
fields, such as medicine for efficacy evaluation of combination
therapies, \citep[e.g.][]{NivolumabIpilimumab2013}, which is a topic
of growing interest \citep{NatureMed2017}.

We first show that the null hypothesis in \eqref{pb:intro} is
equivalent to one without the scalar multipliers, $\sum_{i=1}^{k}
X_i \sim \sum_{i=1}^{h} Y_i$, which simplifies notation (Lemma \ref{lemma:equiv}). To obtain
a test statistic that utilizes all data and therefore outperforms
methods that require equal sized data sets, in Section 2 
we study the maximum likelihood estimator (MLE) for the probability
mass vector (PMV) of $\sum_{i=1}^{k} X_i$. This transpires to be
the discrete convolution of the empirical probability mass vector
(EPMV) of each variable $X_1, \dots, X_k$ and so we refer to it as the ``convolution statistic'' (Proposition \ref{proposition:mle}).

We then derive the asymptotic distribution of the convolution
statistic and build a testing procedure for both goodness-of-fit
and equality in distribution versions (Proposition \ref{proposition:asy}),
leveraging the generalized Wald's method. This technique was
introduced in Moore's work \citep{MooreSpruill1975,Moore1977,MooreMihalko1980,Moore1982}, as an extension of Wald's method
\citep{Wald1943}, to build $\chi^2$ tests for statistics that are
asymptotically normal distributed with a singular covariance matrix.
It was subsequently adjusted in \citep{Hadi1990}, whose version we
employ. Such methodology found applications in the fields of
econometrics \citep{Vuong1987, Andrews1987, Andrews1988,
	WilsonKoehler1991}, biology \citep{Zhang1999,Marchingo2016}, and
statistical theory \citep{Tyler1981,Drost1989,Vonta2008}.  

In Section 
3, we investigate the covariance matrix
rank asymptotic of the convolution statistic (Theorem \ref{thm:rnk}, Corollary \ref{cor:coprime_rnk} and Corollary \ref{cor:general_rnk}), which is the central problem
for the derivation of a testing procedure through the generalized
Wald's framework. Interestingly, such rank is related to the roots
of the probability generating functions of $X_1, \dots, X_k$ and
$Y_1, \ldots, Y_h$ (Lemma \ref{lemma:rnk}).
Finally, we show how a test for sub-independence \citep{Hamedani2013,Schennach2019} can be built from the results achieved above (Corollary \ref{cor:sub_independence}).

To conclude, in Section 4 
we provide simulated performance
analysis for the convolution statistic against Person's $\chi^2$,
and in Section 5 
we discuss the guidelines for its application.
We remark that all the necessary proofs are reported in Appendix. 

\vskip 3mm

\noindent 2. CONVOLUTION STATISTIC

To derive a statistic for the testing of $\mathbf{H}_0\colon a_0 +
\sum_{i=1}^{k} a_i A_i \sim b_0 + \sum_{i=1}^{h} b_i B_i$, we begin
by showing that this null hypothesis is equivalent to another in
which the variables have finite, positive integer support and no
parameters $a_0, \ldots, a_k, b_0, \ldots, b_h$ are present. As a
consequence, we will work in this new setting as it facilitates the
definition of the convolution statistic for the testing of
$\mathbf{H}_0$, especially in regard to a simpler notation.

\begin{lemma}[Null hypothesis simplification]\label{lemma:equiv}
	Let $A_1, \ldots, A_k$, $B_1, \ldots, B_h$ be a sequence of finite and independent random variables that map the probability space $(\Omega, \mathcal{F}, \mathbbm{P})$ into a lattice $\Lambda(\zeta)$ with $\zeta \in \mathbb{R}$. Given the null hypothesis
	\begin{equation}\label{h_0:initial}
	\mathbf{H}_0\colon a_0 + \sum_{i=1}^{k} a_i A_i \sim b_0 + \sum_{i=1}^{h} b_i B_i,
	\end{equation}
	with $a_0, b_0 \in \Lambda(\zeta)$, $a_1, \ldots, a_k, b_1, \ldots, b_h \in \mathbb{Z}$, there exists a sequence of positive, finite and independent random variables $X_1, \ldots, X_k$, $Y_1, \ldots, Y_h$ from the probability space $(\Omega, \mathcal{F}, \mathbbm{P})$ into $\{0, \ldots, r_l\}$, respectively, with $r_l \in \mathbb{N}$ for $l=1, \ldots, k+h$, $\Pro{X_i = 0}>0$ for $i=1, \ldots, k$, $\Pro{Y_j = 0}>0$ for $j=1, \ldots, h$, and such that
	\begin{equation}\label{h_0:reduced}
	\mathbf{H}_0\colon \sum_{i=1}^{k} X_i \sim \sum_{i=1}^{h} Y_i,
	\end{equation}
	is equivalent to \eqref{h_0:initial}.
\end{lemma}

As a result of Lemma \ref{lemma:equiv}, we need only to consider $\mathbf{H}_0$ stated in equation \eqref{h_0:reduced}.
Thus given a sequence of $k \geq 2$ integer and independent random variables
$X_1, \ldots, X_k$ we write, for fixed $i \in \{1, \ldots, k\}$,
that $X_i \sim \vect{x}_i \in \Delta^{r_i}$ with
\begin{equation*}
\Delta^{r_i} = \left\{\vect{v} = (v_0, \ldots, v_{r_i}) \in \mathbb{R}^{r_i +1}\colon v_{0}, v_{r_i} \!\! \in (0,1) ;\, v_j \geq 0 , j=1, \ldots, r_i-1 ; \sum_{j=1}^{r_i} v_j =1  \right\}
\end{equation*}
to indicate that $X_i$ takes values in $\{0, \ldots, r_i\} \subseteq \mathbb{N} \cup \{0\}$, with $r_i >0$, and is distributed with PMV $\vect{x}_i = (x_{i0}, \ldots, x_{ir_i})$, that is $\Pro{X_i = j} = x_{ij}$ for $j=0,\ldots, r_i$.

We remark that $x_{i0}, x_{ir_i} \in (0,1)$, for every $i = 1, \ldots, k$, are assumed to avoid degenerate cases, without loss of generality.
In fact, $x_{i0} > 0$ descends form Lemma \ref{lemma:equiv}.
Moreover, given any $t \leq k$ such that $x_{tr_t} = \Pro{X_t = r_t} = 0$, there exists $\tau = \max \{j \colon x_{tj} > 0 \} < r_t$, so that $X_t$ can be replaced by $\tilde{X}_t \sim \tilde{\vect{x}}_t = (\tilde{x}_{t0}, \ldots, \tilde{x}_{t\tau}) \in \Delta^\tau$, with $\tilde{x}_{tj} = x_{tj}$ for $j=0, \ldots, \tau$. Lastly, the constraint $x_{i0}, x_{ir_i} < 1$ for every $i = 1, \ldots, k$ ensures that $X_i$ is not a constant value.

From now on, with these assumptions and notation, for the null hypothesis of goodness-of-fit test we consider
\begin{equation}\label{h0:gof}
\mathbf{H}_0\colon \sum_{i=1}^{k} X_i \sim \vect{z}
\end{equation}
with $s = \sum_{i=1}^k r_i$ and $\vect{z} \in \Delta^{s}$.
By independence, the sum of $X_1, \ldots, X_k$ is distributed as the discrete convolution, denoted $\ast$, of their PMVs, that is
\begin{equation*}
\sum_{i=1}^{k} X_i \sim \vect{x}_1 \ast \ldots \ast \vect{x}_k,
\end{equation*}
where, for any two vectors $\vect{v} = (v_0, \ldots, v_a) \in \mathbb{R}^{a + 1}$, $\vect{w} = (w_0, \ldots, w_b) \in \mathbb{R}^{b + 1}$ with $a,b>0$, we have $\vect{v} \ast \vect{w} \in \mathbb{R}^{a+b+1}$ and $(\vect{v} \ast \vect{w})_i = \sum_{j=0}^{a} \sum_{l=0}^{b} v_j w_l \delta_{j+l,i}$ with $\delta_{i,j} = 1$ if $i=j$ and being null otherwise.
Hence, the null hypothesis for the goodness-of-fit-test \eqref{h0:gof} is equivalent to
\begin{equation*}
\mathbf{H}_0\colon \vect{x}_1 \ast \ldots \ast \vect{x}_k = \vect{z}.
\end{equation*}
Our first goal is to determine a statistic to test \eqref{h0:gof}, which will subsequently be extended
to assess the equality in distribution between $\sum_{i=1}^k X_i $ and $ \sum_{i=1}^h Y_i$, where $Y_1 \sim \vect{y}_1 \in \Delta^{r_{k+1}}, \ldots, Y_h \sim \vect{y}_h \in \Delta^{r_{k+h}}$ are other $h \geq 1$ independent random variables, with $s=\sum_{i=1}^k r_i = \sum_{i=1}^h r_{k+i}$, namely
\begin{equation}\label{h0:tfh}
\mathbf{H}_0\colon \, \sum_{i=1}^k X_i \sim \sum_{i=1}^h Y_i,
\end{equation}
or, equivalently,
\begin{equation*}
\mathbf{H}_0\colon \, \vect{x}_1 \ast \ldots \ast \vect{x}_k = \vect{y}_1 \ast \ldots \ast \vect{y}_h.
\end{equation*}

When the PMVs $\vect{x}_i$ for $i=1,\ldots, k$ and $\vect{y}_j$ for
$j=1,\ldots, h$ are unknown, care must be taken to define the test
statistics for \eqref{h0:gof} and \eqref{h0:tfh} based only on
available information. In regard, the data consist of the observation
of $n_i$ independent random variables $\{X_{i1}, \ldots, X_{i n_i}\}$
identically distributed as $X_i$ for $i=1, \ldots, k$ and $n_{k+i}$
independent random variables $\{Y_{i1}, \ldots, Y_{i n_{k+i}}\}$
identically distributed as $Y_i$ for $i=1, \ldots, h$.
Following a nonparametric approach, we fix $i \in \{1, \ldots, k\}$ and define $\hat{\vect{x}}_{in_i} \in \Delta^{r_i}$ the MLE of $\vect{x}_i$, that is
\begin{equation*}
(\hat{\vect{x}}_{in_i})_u = \frac{1}{n_i} \sum_{j=1}^{n_i} \mathbbm{1}_{\{X_{ij}=u\}}
\end{equation*}
for $u=0, \ldots, r_i$, with $\mathbbm{1}_{A}$ being the indicator function of the event $A$.
In particular, by the multivariate central limit theorem \citep{Serfling1980}, 
$\sqrt{n_i}(\hat{\vect{x}}_{in_i} - \vect{x}_i)$ is asymptotically distributed as a centered normal random variable with covariance $\Smat{\vect{x}_i}$ for $n_i$ large, namely
$\sqrt{n_i}(\hat{\vect{x}}_{in_i} - \vect{x}_i) \sim_{n_i \to \infty} \mathcal{N}(\Smat{\vect{x}_i})$,
where, for any PMV $\vect{v} = (v_0, \ldots, v_a) \in \Delta^a$ and $a\geq 0$, we define $\Smat{\vect{v}} \in \mathbb{R}^{a+1}\times\mathbb{R}^{a+1}$ such that $(\Smat{\vect{v}})_{ij} = v_i \delta_{i,j} - v_i v_j$ for $i,j=0, \ldots, a$.
With this notation, we derive the MLE for the distribution of $\sum_{i=1}^k X_i$.
\begin{proposition}[MLE for a sum of independent random variables]\label{proposition:mle}
	Given $k\geq 2$, let $\{X_{i1}, \ldots, X_{in_i}\}$ be $n_i \in \mathbb{N}$ random variables independent and identically distributed as $X_i \sim \vect{x}_i \in \Delta^{r_i}$ with $r_i \in \mathbb{N}$, for $i=1, \ldots, k$.
	Set $s = \sum_{i=1}^k r_i$. The MLE for the PMV $\vect{x}_1 \ast \ldots \ast \vect{x}_k \in \Delta^s$ of $\sum_{i=1}^k X_i$ is $\hat{\vect{x}}_{1n_1} \ast \ldots \ast \hat{\vect{x}}_{kn_k} \in \Delta^{s}$, defined as
	\begin{equation*}
	(\hat{\vect{x}}_{1n_1} \ast \ldots \ast \hat{\vect{x}}_{kn_k})_u = \left(\prod_{j=1}^{k} n_j\right)^{-1} \sum_{i_1 = 1}^{n_1} \cdots \sum_{i_k = 1}^{n_k} \mathbbm{1}_{\{\sum_{j=1}^{k} X_{ji_j}=u\}},
	\end{equation*}
	for every $u=0, \ldots, s$.
\end{proposition}

Proposition \ref{proposition:mle} shows that the MLE for $\vect{x}_1 \ast \ldots \ast \vect{x}_k$ is calculated using all observables $X_{ij}$ for $j=1, \ldots, n_i$ and $i=1, \ldots, k$, which may not be the case for Person's statistic, as explained in Section 
1, if at least one of the sample sizes $n_1, \ldots, n_k$ differs from another.

We introduce additional notation for what follows:
$\matr{A}^+$, $\matr{A}^\prime$, $\Ker{\matr{A}}$, $\Nul{\matr{A}}$, $\Rnk{\matr{A}}$ for the Moore-Penrose inverse, transpose, kernel, nullity and rank of a matrix $\matr{A}$, respectively \citep{HornJohnson1986, Cstar2000};
$\chi^2(s)$ to indicate the $\chi^2$ distribution with $s>0$ degrees of freedom;
$\matr{T}^{b+1}(v) \in \mathbb{R}^{b+1}\times\mathbb{R}^{a+b+1}$ for the matrix of the discrete convolution between $\vect{v} \in \mathbb{R}^{a+1}$ and any $b\!+\! 1$-dimensional vector, i.e. $\matr{T}^{b+1}(\vect{v}) \vect{w} = \vect{v} \ast \vect{w} \in \mathbb{R}^{a+b+1}$ with $\vect{w} \in \mathbb{R}^{b+1}$, given $a,b \geq 0$. We write $\Tmat{\vect{v}}$ without explicit domain dimension if this is clear from the context.

We are now ready to determine the asymptotic behavior of
$\hat{\vect{x}}_{1n_1} \ast \ldots \ast \hat{\vect{x}}_{kn_k}$,
which follows from an application of the delta method as well as properties
of quadratic transformation of asymptotically multivariate normal
vectors \citep{Serfling1980}. In order for the MLE $\hat{\vect{x}}_{1n_1}
\ast \ldots \ast \hat{\vect{x}}_{kn_k}$ to converge to $\vect{x}_1
\ast \ldots \ast \vect{x}_k$ it is necessary that the sample sizes
$n_1, \ldots, n_k$ grow with proportional rates. For this reason,
from now on, we set 
\begin{equation*}
m = \min(n_1, \ldots, n_{k+h}) \text{ and assume }
c_i = \lim_{m \to \infty} \frac{m}{n_i}
\end{equation*}
is finite and positive for every $i=1, \ldots, k+h$.
\begin{proposition}[Asymptotic normality of convolutions]\label{proposition:asy}
	Under the null hypothesis \eqref{h0:gof}, that $\vect{x}_1 \ast \ldots \ast \vect{x}_k = \vect{z}$, it holds that
	\begin{equation}\label{v_n}
	\vect{V}_m = \sqrt{m} \left( \hat{\vect{x}}_{1n_1} \ast \ldots \ast \hat{\vect{x}}_{kn_k} - \vect{z} \right) \Asy{m}{\infty} \mathcal{N}(\matr{\Psi})
	\end{equation}
	and
	\begin{equation}\label{asy_chi:gof}
	\vect{V}_m^\prime \matr{\Psi}^{+} \vect{V}_m \Asy{m}{\infty} \chi^2\left( \Rnk{\matr{\Psi}} \right)
	\end{equation}
	where $\matr{\Psi} = \sum_{i=1}^k c_i \Tmat{\vect{x}_{(i)}} \Smat{\vect{x}_i} \Tmat{\vect{x}_{(i)}}^\prime$ and $\vect{x}_{(i)} = \vect{x}_1 \ast \ldots \ast \vect{x}_{i-1} \ast \vect{x}_{i+1} \ast \ldots \ast \vect{x}_k$ for $i=1, \ldots, k$.
	Alternatively, under the null hypothesis \eqref{h0:tfh}, that $\vect{x}_1 \ast \ldots \ast \vect{x}_k = \vect{y}_1 \ast \ldots \ast \vect{y}_h $, it holds that
	\begin{equation}\label{w_n}
	\vect{W}_m = \sqrt{m} \left( \hat{\vect{x}}_{1n_1} \ast \ldots \ast \hat{\vect{x}}_{kn_k} - \hat{\vect{y}}_{1n_{k+1}} \ast \ldots \ast \hat{\vect{y}}_{hn_{k+h}} \right) \Asy{m}{\infty} \mathcal{N}(\matr{\Psi} + \matr{\Xi})
	\end{equation}
	and
	\begin{equation}\label{asy_chi:tfh}
	\vect{W}_m^\prime (\matr{\Psi} + \matr{\Xi})^{+} \vect{W}_m \Asy{m}{\infty} \chi^2\left( \Rnk{\matr{\Psi} + \matr{\Xi}} \right)
	\end{equation}
	where $\matr{\Xi} = \sum_{i=1}^h c_{k+i} \Tmat{\vect{y}_{(i)}} \Smat{\vect{y}_i} \Tmat{\vect{y}_{(i)}}^\prime$ and $\vect{y}_{(i)} = \vect{y}_1 \ast \ldots \ast \vect{y}_{i-1} \ast \vect{y}_{i+1} \ast \ldots \ast \vect{y}_k$ for $i=1, \ldots, h$.
\end{proposition}
We remark that expressions \eqref{asy_chi:gof} and \eqref{asy_chi:tfh} require the knowledge of $\matr{\Psi}$ and $\matr{\Psi}+ \matr{\Xi}$, but these may, in general, be unknown.
Thus we take advantage of the generalized Wald's method \citep{Moore1977}, which shows how to construct $\chi^2$ tests from consistent estimators of the covariance variance matrices such as $\matr{\Psi}$ and $\matr{\Psi}+\matr{\Xi}$. We recall here \citet[][Theorem 2]{Moore1977} which will serve as backbone for the subsequent results.
\begin{proposition}[Generalized Wald's method; {\cite{Moore1977}, Theorem 2}]\label{thm:moore}
	Suppose a sequence of estimators $\{\hat{\vect{\theta}}_m\}_{m\geq1}$ of a parameter $\vect{\theta}_0 \in \mathbb{R}^{d}$, with $d>0$, is such that
	\begin{equation*}
	\sqrt{m}\left( \hat{\vect{\theta}}_m - \vect{\theta}_0 \right) \Asy{m}{\infty} \mathcal{N}(\matr{\Sigma})
	\end{equation*}
	with $\Rnk{\matr{\Sigma}} \leq d$. Noted $\{\matr{B}_m\}_{m\geq1}$ a sequence of $d$-dimensional square matrices such that $\matr{B}_m \sim_{m \to \infty} \matr{B}$ with $\matr{B}$ generalized-inverse of $\matr{\Sigma}$, then
	\begin{equation*}
	m\left( \hat{\vect{\theta}}_m - \vect{\theta}_0 \right)^\prime \matr{B}_m \left( \hat{\vect{\theta}}_m - \vect{\theta}_0 \right) \Asy{m}{\infty} \chi^2(\Rnk{\matr{\Sigma}}).
	\end{equation*}
\end{proposition}
Since the entries of
\begin{equation*}
\hat{\matr{\Psi}}_m = \sum_{i=1}^k c_i \Tmat{\hat{\vect{x}}_{(i)n_i}}  \Smat{\hat{\vect{x}}_{in_i}} \Tmat{\hat{\vect{x}}_{(i)n_i}}^\prime
\end{equation*}
are continuous functions of $\vect{x}_1, \ldots, \vect{x}_k$, whose consistent estimator are $\hat{\vect{x}}_{1n_1}, \ldots, \hat{\vect{x}}_{kn_k}$ respectively, then $\hat{\matr{\Psi}}_m$ is a consistent estimator of $\matr{\Psi}$ and similarly
\begin{equation*}
\hat{\matr{\Xi}}_m =  \sum_{i=1}^h c_{k+i} \Tmat{\hat{\vect{y}}_{(i)n_{k+i}}}  \Smat{\hat{\vect{y}}_{in_{k+i}}} \Tmat{\hat{\vect{y}}_{(i)n_{k+i}}}^\prime
\end{equation*}
for $\matr{\Xi}$.

Note that Proposition \ref{thm:moore} cannot be directly applied
to \eqref{v_n} by setting $\matr{B}_m = \hat{\matr{\Psi}}_m^+$, as
$\hat{\matr{\Psi}}_m^+$ may not be a consistent estimator of
$\matr{\Psi}^+$. Given a sequence of consistent estimators
$\{\matr{A}_m\}_{m \geq 1}$ for a matrix $\matr{A}$ of finite
dimensions, then $\{\matr{A}^+_m\}_{m\geq1}$ is a sequence of
consistent estimators for $\matr{A}^+$ if and only if $\Rnk{\matr{A}_m}
= \Rnk{\matr{A}}$ for $m$ large \citep{Nashed1976:325}.  In particular,
as the rank is a lower-semicontinuous operator on the space of
finite dimensional matrices, then only
$\operatorname{rk}(\hat{\matr{\Psi}}_m) \geq \Rnk{\matr{\Psi}}$ is
guaranteed as $m$ tends to infinity.

If the limiting rank is known, consistency is ensured by Eckart-Young-Mirsky's theorem \citep{Eckart1936}, which is the solution to the basic low rank approximation of a finite dimensional matrix \citep{Markovsky2012}.
To this end, given any $d$-dimensional symmetric matrix $\matr{A} \in \mathbb{R}^{d} \times \mathbb{R}^{d}$ with $d\in \mathbb{N}$ and its eigendecomposition $\matr{A} = \matr{P}^\prime \matr{\Lambda} \matr{P}$, with $0< \Rnk{\matr{A}} \leq d$, $\matr{\Lambda}$ diagonal matrix of the decreasing eigenvalues and $\matr{P}$ orthogonal matrix, then for any $0< r \leq \Rnk{\matr{A}}$ we define a rank-$r$ matrix that approximates $\matr{A}$ (in light of the Eckart-Young-Mirsky theorem) as
$\matr{A}^r = ( \matr{D}^r \matr{P})^\prime \matr{\Lambda}^r  \matr{D}^r \matr{P} \in \mathbb{R}^{d} \times \mathbb{R}^{d}$,
where $\matr{D}^r$ is a $\mathbb{R}^r \times \mathbb{R}^{d}$ matrix with $1$ at the diagonal and $0$ elsewhere and $\matr{\Lambda}^r$ is the $\mathbb{R}^r \times \mathbb{R}^r$ diagonal matrix of the largest $r$ eigenvalues of $\matr{\Sigma}$. In particular, $\matr{A}^r$ may not be unique, as for the case when the $r^{\text{th}}$ and $r+1^{\text{th}}$ eigenvalues are equal.
The following result is found in \citet[Theorem 2.3]{Hadi1990} for the generalized inverses and we report it here for the case of Moore-Penrose inverses.
\begin{proposition}[Rank approximation; {\citet{Hadi1990}, Theorem 2.3}]\label{thm:rkred}
	Suppose a sequence of centered random variables $\{\vect{U}_m\}_{m\geq1} \in \mathbb{R}^{d}$ is asymptotically distributed as $\mathcal{N}(\matr{\Sigma})$ for $m$ large, with $0 < \Rnk{\matr{\Sigma}} \leq d$ where $d>0$.
	Let $\{\hat{\matr{\Sigma}}_m\}_{m\geq1}$ be a sequence of square matrices that are consistent estimators of $\matr{\Sigma}$, then for every $0 < r \leq \Rnk{\matr{\Sigma}}$ 
	\begin{equation*}
	\vect{U}_m^\prime (\hat{\matr{\Sigma}}_m^r)^+ \vect{U}_m \Asy{m}{\infty} \chi^2(r),
	\end{equation*}
	where $\hat{\matr{\Sigma}}_m^r$ is a rank-$r$ approximation of $\hat{\matr{\Sigma}}_m$.
\end{proposition}
Proposition \ref{thm:rkred} highlights the central role of the rank of $\matr{\Sigma}$, which will be derived in the next section for $\matr{\Sigma} = \matr{\Psi}$ and $\matr{\Sigma} = \matr{\Psi} + \matr{\Xi}$. In general, the determination of $\Rnk{\matr{\Sigma}}$ may be a difficult problem that depends on the structure of the $\matr{\Sigma}$ under consideration, and this limitation may explain why an otherwise flexible tool such as the generalized Wald's method from Proposition $\ref{thm:moore}$ is not more widely employed.
But, if the rank is known, Proposition \ref{thm:rkred} provides a method for statistical testing null hypotheses, such as $\mathbf{H}_0\colon \, \vect{x}_1 \ast \ldots \ast \vect{x}_k = \vect{z}$ and $\mathbf{H}_0\colon \, \vect{x}_1 \ast \ldots \ast \vect{x}_k = \vect{y}_1 \ast \ldots \ast \vect{y}_h$, under which $\matr{\Sigma}$ is not invertible.
This result also assures a solution if only a lower bound of the rank is given, at the cost of statistical power.
Furthermore, the exclusion of smaller eigenvalues may still be necessary to achieve numerical stability when calculating the pseudo-inverse of $\hat{\matr{\Sigma}}_m$, as due to Proposition \ref{thm:rkred}, the effect of that truncation can be accounted for in the statistic formulation.
\vskip 3mm

\noindent 3. DETERMINING THE COVARIANCE MATRIX RANK

In this section, we investigate the rank of
$\matr{\Psi}$ and $\matr{\Psi} + \matr{\Xi}$, the covariance matrices from \eqref{v_n} and \eqref{w_n} of Proposition \ref{proposition:asy}, in order to derive the number of degrees of freedom from the limiting statistics for the goodness-of-fit  \eqref{asy_chi:gof} and equality in distribution \eqref{asy_chi:tfh} tests.
Focusing on $\matr{\Psi}$, we begin by showing that $c_i \Tmat{\vect{x}_{(i)}} \Smat{\vect{x}_i} \Tmat{\vect{x}_{(i)}}^\prime$ is a positive semidefinite matrix for any fixed $i \in \{1, \ldots, k\}$. In fact, since $\Smat{\vect{x}_i}$ is positive semidefinite, for every $\vect{v} \in \mathbb{R}^{s+1}$
\begin{equation}\label{TST_psd}
c_i \vect{v}^\prime \Tmat{\vect{x}_{(i)}} \Smat{\vect{x}_i} \Tmat{\vect{x}_{(i)}}^\prime \vect{v} = c_i \vect{w}^\prime \Smat{\vect{x}_i} \vect{w} \geq 0,
\end{equation}
where $\vect{w} = \Tmat{\vect{x}_{(i)}}^\prime \vect{v}$.
Additionally, we deduce from \citet[Observation 7.1.3]{HornJohnson1986} that given $\matr{A}$ and $\matr{B}$ two positive semidefinite matrices with the same dimensions, then
\begin{equation}\label{prop_sum_psd_mat}
\Ker{\matr{A}+\matr{B}} = \Ker{\matr{A}} \cap \Ker{\matr{B}}.
\end{equation}
Taken together, \eqref{TST_psd} and \eqref{prop_sum_psd_mat}
imply
\begin{equation}\label{step:ker_psi}
\Ker{\matr{\Psi}} = \bigcap_{i=1}^k \Ker{\Tmat{\vect{x}_{(i)}}  \Smat{\vect{x}_i} \Tmat{\vect{x}_{(i)}}^\prime}.
\end{equation}
Using kernel properties, we write
\begin{equation}\label{ker_STi1}
\Ker{\Tmat{\vect{x}_{(i)}} \Smat{\vect{x}_i} \Tmat{\vect{x}_{(i)}}^\prime}
= \Ker{\Tmat{\vect{x}_{(i)}}^\prime} \oplus \{\vect{v} \in \mathbb{R}^{s+1} \colon \Tmat{\vect{x}_{(i)}}^\prime \vect{v} \in \Ker{\Smat{\vect{x}_i}} \}
\end{equation}
where
$\oplus$ represents the direct sum operation.

Let $i \in \{1, \ldots, k\}$ be fixed and let $L_i = \{l \colon x_{il}=0\}$ be the set of indexes of the null entries of $\vect{x}_i \in \Delta^{r_i}$.
In general, the kernel of $\Smat{\vect{x}_i}$ is generated by the $r_{i}\!+\!1$-dimensional all-ones vector $\vect{1}_{r_{i}}$ and the canonical vectors $\vect{e}^i_l = (e^i_{l0}, \ldots, e^i_{lr_i}) \in \mathbb{R}^{r_i + 1}$ for every $l \in L_i$, where $e^i_{lu} = \delta_{l,u}$ for $u=0, \ldots, r_i$, that is
$\Ker{\Smat{\vect{x}_i}} = \langle \vect{1}_{r_i} \rangle \oplus E_i$,
where $E_i = \langle \{\vect{e}^i_l \colon l \in L_i\}\rangle$.

Denoting $r_{(i)} = \sum_{j\neq i}^k r_j$, we can expand $\Tmat{\vect{x}_{(i)}}$ into
\begin{equation*}
\Tmat{\vect{x}_{(i)}} = 
\begin{bmatrix}
x_{(i)0} & 0 & \dots  & 0 \\
x_{(i)1} & x_{(i)0} & \ddots & \vdots \\
\vdots & \vdots & \ddots  & 0 \\
x_{(i)r_{(i)}} & x_{(i)r_{(i)}-1} &\ddots  & x_{(i)0}\\
0 & x_{(i)r_{(i)}} & \ddots  & \vdots \\
\vdots & \ddots & \ddots  & \vdots \\
0 & \dots  & 0 & x_{(i)r_{(i)}}
\end{bmatrix},
\end{equation*}
from which we deduce
$\Tmat{\vect{x}_{(i)}}^\prime \vect{1}_{s} = \vect{1}_{r_{i}}$,
for every $\vect{x}_{(i)} \in \Delta^{r_{(i)}}$, and in particular $\vect{1}_{s} \notin \Ker{\Tmat{\vect{x}_{(i)}}^\prime}$.

To achieve an explicit formulation for the rank of $\matr{\Psi}$ (and analogously for $\matr{\Psi} + \matr{\Xi}$), in Theorem \ref{thm:rnk} we will assume that $\vect{x}_i \in \Delta^{r_i}_{\text{Int}} \subset \Delta^{r_i}$, defined as
\begin{equation*}
\Delta^{r_i}_{\text{Int}} = \{\vect{v} = (v_0, \ldots, v_{r_i}) \in \mathbb{R}^{r_i + 1} \colon \sum_{j=1}^{r_i} v_j = 1; 0<v_j <1, j=0, \ldots, r_i \}
\end{equation*}
for every $i=1, \ldots, k$. This ensures that $E_i = \emptyset$, so that
$\Ker{\Smat{\vect{x}_i}} = \langle \vect{1}_{r_i} \rangle$.
Under this hypothesis, from \eqref{step:ker_psi} and \eqref{ker_STi1} we deduce that
\begin{equation}\label{ob:kerpsi}
\Ker{\matr{\Psi}} = \bigcap_{i=1}^k \Ker{\Tmat{\vect{x}_{(i)}}  \Smat{\vect{x}_i} \Tmat{\vect{x}_{(i)}}^\prime} = \langle \vect{1}_{s} \rangle \oplus \bigcap_{i=1}^k \Ker{\Tmat{\vect{x}_{(i)}}^\prime},
\end{equation}
and, with the same reasoning applied to $\matr{\Psi} + \matr{\Xi}$, it follows that
\begin{equation}\label{ob:kerpsixi}
\Ker{\matr{\Psi}+ \matr{\Xi}} = \langle \vect{1}_{s} \rangle \oplus \bigg( \big( \bigcap_{i=1}^{k} \Ker{\Tmat{\vect{x}_{(i)}}^\prime} \big) \cap \big( \bigcap_{j=1}^{h} \Ker{\Tmat{\vect{y}_{(j)}}^\prime} \big) \bigg).
\end{equation}

In the following Lemma we show how
\begin{equation}\label{kernel_cap_x}
\bigcap_{i=1}^k \Ker{\Tmat{\vect{x}_{(i)}}^\prime}
\end{equation}
and
\begin{equation}\label{kernel_cap_xy}
\bigg(\bigcap_{i=1}^{k} \Ker{\Tmat{\vect{x}_{(i)}}^\prime} \bigg) \cap \bigg(\bigcap_{j=1}^{h} \Ker{\Tmat{\vect{y}_{(j)}}^\prime} \bigg)
\end{equation}
depend on the roots in common between the probability generating functions of the random variables $X_1, \ldots, X_k, Y_1, \ldots, Y_h$.
This result will be achieved in full generality without restrictions for the PMVs, that is with $\vect{x}_1 \in \Delta^{r_1}, \ldots, \vect{x}_k \in \Delta^{r_k}, \vect{y}_1 \in \Delta^{r_{k+1}}, \ldots, \vect{y}_h \in \Delta^{r_{k+h}}$.
We first provide some insight into this connection by considering the case $k=2$
\begin{equation*}
\Ker{\Tmat{\vect{x}_{(1)}}^\prime} \cap \Ker{\Tmat{\vect{x}_{(2)}}^\prime} = \Ker{\begin{bmatrix} \Tmat{\vect{x}_{2}} & \Tmat{\vect{x}_{1}} \end{bmatrix}^\prime}.
\end{equation*}
In fact, $\begin{bmatrix} \Tmat{\vect{x}_{2}} & \Tmat{\vect{x}_{1}} \end{bmatrix} \in \mathbb{R}^{r_1 + r_2 + 2} \times \mathbb{R}^{r_1 + r_2 + 1}$ has the same structure, with different dimensions, of a Sylvester matrix \citep{Markovsky2012}, whose nullity is the degree of the polynomial from the greatest common divisor of the probability generating functions associated to $\vect{x}_1$ and $\vect{x}_2$.

To formalize the connection with PMVs and polynomials,
we introduce the bijection $\varphi\colon \cup_{a\geq0} \{\vect{u} = (u_0, \ldots, u_a)\in \mathbb{R}^{a+1} \colon \sum_{i=0}^{a}u_i = 1; u_a \neq 0\} \rightarrow \{ u(t)\in \mathbb{R}[t] \colon u(1)=1\}$ that, for any $a \geq 0$, maps a vector $\vect{v}=(v_0, \ldots, v_a) \in \{\vect{u}=(u_0, \ldots, u_a)\in \mathbb{R}^{a+1} \colon \allowbreak \sum_{i=0}^{a}u_i = 1; u_a \neq 0\}$ to the polynomial $\varphi(\vect{v})(t) = \sum_{i=0}^a v_i t^i \in \mathbb{R}[t]$ of degree $\deg\varphi(\vect{v}) = a$ with coefficients $\vect{v}$. In particular, this map transforms the convolution of vectors into the product of polynomials: given $\vect{w} \in \{\vect{u}=(u_0, \ldots, u_b)\in \mathbb{R}^{b+1} \colon \sum_{i=0}^{b} u_i = 1; u_b \neq 0\}$, for any $b\geq0$, we have
\begin{equation*}
\varphi(\vect{v} \ast \vect{w})(t) = \sum_{i=0}^a (\vect{v} \ast \vect{w})_i t^i = \varphi(\vect{v})(t)\varphi(\vect{w})(t).
\end{equation*}
The map $\varphi$ allows the extension of the notion of greatest common divisor between any two polynomials $\gcd(\varphi(\vect{v})$, $\varphi(\vect{w}))$ to their related vectors $\vect{v}$, $\vect{w}$. This is achieved by establishing $\gcd(u_1(t), u_2(t)) \in \{ u(t)\in \mathbb{R}[t] \colon u(1)=1\}$ for any $u_1(t), u_2(t) \in \{ u(t)\in \mathbb{R}[t] \colon u(1)=1\}$, so that the greatest common divisor is uniquely defined, and by setting, for any $\vect{v},\vect{w} \in \cup_{a\geq0} \{\vect{u}= (u_0, \ldots, u_a)\in \mathbb{R}^{a+1} \colon \sum_{i=0}^{a}u_i = 1; u_a \neq 0\}$,
\begin{equation*}
\gcd(\vect{v},\vect{w}) = \varphi^{-1}(\gcd(\varphi(\vect{v})(t), \varphi(\vect{w})(t)))
\in \{\vect{u}= (u_0, \ldots, u_{r_g})\in \mathbb{R}^{r_g+1} \colon \sum_{i=0}^{r_g}u_i = 1; u_{r_g} \neq 0\}
\end{equation*}
with $r_g = \deg\gcd(\varphi(\vect{v})(t), \varphi(\vect{w})(t)) \geq 0$.
In particular, we say the vectors $\vect{v}$, $\vect{w}$ are coprime if and only if $r_g = 0$. Following the same logic, we import the concept of least common multiple between $\vect{v}$ and $\vect{w}$, denoted $\lcm(\vect{v},\vect{w})$, and the property of divisibility between vectors.
Of note, with the notation above, the probability generating functions of $\vect{x}_1, \ldots, \vect{x}_k, \vect{y}_1, \ldots, \vect{y}_h$ are, respectively, $\varphi(\vect{x}_1), \ldots, \varphi(\vect{x}_k), \varphi(\vect{y}_1), \ldots, \varphi(\vect{y}_h)$.

We now establish the relation between the kernels of \eqref{kernel_cap_x}, \eqref{kernel_cap_xy} and the greatest common divisors $\vect{g}_k = \gcd(\vect{x}_{(1)}, \ldots, \vect{x}_{(k)})$ and $\bar{\vect{g}}_h = \gcd(\vect{y}_{(1)}, \ldots, \vect{y}_{(h)})$.

\begin{lemma}[Kernels from gcd of PMVs]\label{lemma:rnk}
	Let $k\geq 2$ and $\vect{x}_1 \in \Delta^{r_1}, \ldots, \vect{x}_k \in \Delta^{r_k}$.
	Given $\vect{g}_k = \gcd(\vect{x}_{(1)}, \ldots, \vect{x}_{(k)}) \in \mathbb{R}^{r_{g_k} +1}$, it holds that $\Tmat{\vect{g}_k}^\prime \in  \mathbb{R}^{\sum_{i=1}^k r_i - r_{g_k}+1} \times \mathbb{R}^{\sum_{i=1}^k r_i+1}$ and
	\begin{equation}\label{ker:psi}
	\bigcap_{i=1}^{k} \Ker{\Tmat{\vect{x}_{(i)}}^\prime} = \Ker{\Tmat{\vect{g}_k}^\prime}.
	\end{equation}
	Additionally, let $h \geq 1$ and $\vect{y}_1 \in \Delta^{r_{k+1}}, \ldots, \vect{y}_h \in \Delta^{r_{k+h}}$.
	Given $\bar{\vect{g}}_h = \gcd(\vect{y}_{(1)}, \ldots, \vect{y}_{(h)}) \in \mathbb{R}^{r_{\bar{g}_h} +1}$ and $\tilde{\vect{g}} = \gcd(\vect{g}_k, \bar{\vect{g}}_h) \!\! \in \! \mathbb{R}^{r_{\tilde{g}} +1}$, it holds that $\Tmat{\bar{\vect{g}}_h}^\prime \!\! \in \! \mathbb{R}^{\sum_{i=1}^k r_i - r_{\bar{g}_h}+1} \times \mathbb{R}^{\sum_{i=1}^k r_i+1}$, $\Tmat{\tilde{\vect{g}}}^\prime \!\! \in\! \mathbb{R}^{\sum_{i=1}^k r_i - r_{\tilde{g}}+1} \times \mathbb{R}^{\sum_{i=1}^k r_i+1}$ and
	\begin{equation}\label{ker:psi_xi}
	\bigcap_{i=1}^{k} \Ker{\Tmat{\vect{x}_{(i)}}^\prime} \cap \bigcap_{j=1}^{h} \Ker{\Tmat{\vect{y}_{(j)}}^\prime} = \Ker{\begin{bmatrix}\Tmat{\vect{g}_k}^\prime \\ \Tmat{\bar{\vect{g}}_h}^\prime \end{bmatrix}} = \Ker{\Tmat{\tilde{\vect{g}}}^\prime}
	\end{equation}
\end{lemma}

As consequence of Lemma \ref{lemma:rnk}, we can finally determine the rank of the covariance matrices $\matr{\Psi}$ and $\matr{\Psi}+\matr{\Xi}$ assuming $\vect{x}_1 \in \Delta^{r_{1}}_{\text{Int}}, \ldots, \vect{x}_k \in \Delta^{r_{k}}_{\text{Int}}, \vect{y}_1 \in \Delta^{r_{k+1}}_{\text{Int}}, \ldots, \vect{y}_h \in \Delta^{r_{k+h}}_{\text{Int}}$, in order to calculate the number of degrees of freedom of the limiting $\chi^2$ distribution in \eqref{asy_chi:gof} and \eqref{asy_chi:tfh} from Proposition \ref{proposition:asy} for this case.
\begin{theorem}[Covariance matrix rank]\label{thm:rnk}
	Under the assumptions of Proposition \ref{proposition:asy} and the notations of Lemma \ref{lemma:rnk}, and given $\vect{x}_1 \in \Delta^{r_{1}}_{\textnormal{Int}}, \ldots, \vect{x}_k \in \Delta^{r_{k}}_{\textnormal{Int}}, \vect{y}_1 \in \Delta^{r_{k+1}}_{\textnormal{Int}}, \ldots, \vect{y}_h \in \Delta^{r_{k+h}}_{\textnormal{Int}}$, it follows that
	\begin{equation}\label{thm_kerpsi}
	\Ker{\matr{\Psi}} = \langle \vect{1}_{s} \rangle \oplus \Ker{\Tmat{\vect{g}_k}^\prime}
	\end{equation}
	and
	\begin{equation}\label{thm_kerpsixi}
	\Ker{\matr{\Psi} + \matr{\Xi}} = \langle \vect{1}_{s} \rangle \oplus \Ker{\Tmat{\tilde{\vect{g}}}^\prime}.
	\end{equation}
	In particular, with $s$  defined immediately prior to equation \eqref{h0:tfh},
	\begin{equation}\label{rank_psi}
	\Rnk{\matr{\Psi}} = s - r_{g_k}
	\end{equation}
	and
	\begin{equation}\label{rank_psi_xi}
	\Rnk{\matr{\Psi}+\matr{\Xi}} = s -r_{\tilde{g}}.
	\end{equation}
\end{theorem}

In the case where $\vect{x}_1, \ldots, \vect{x}_k, \vect{y}_1, \ldots, \vect{y}_h$ are coprime, that is when their probability generating functions $\varphi(\vect{x}_1), \ldots, \varphi(\vect{x}_k), \varphi(\vect{y}_1), \ldots, \varphi(\vect{y}_h)$ have no root in common, we can simplify Theorem \ref{thm:rnk} as follows.
\begin{corollary}[Rank from the coprime case]\label{cor:coprime_rnk}
	Under the assumptions of Theorem \ref{thm:rnk}, if $\vect{x}_1,
	\ldots, \vect{x}_k, \vect{y}_1, \ldots, \vect{y}_h$ are
	coprime vectors, then $\Rnk{\matr{\Psi}} =
	\Rnk{\matr{\Psi}+\matr{\Xi}} = s$, where $s$  is defined
	immediately prior to equation \eqref{h0:tfh}.
\end{corollary}

A statistic for the goodness-of-fit \eqref{h0:gof} and equality in distribution \eqref{h0:tfh} tests can also be built for the general case where $\vect{x}_i \in \Delta^{r_{i}}$ and $\vect{y}_i \in \Delta^{r_{i+k}}$, leveraging Proposition \ref{thm:rkred} and the following lower bound for the covariance matrix rank.
\begin{corollary}[Rank lower bound for the general case]\label{cor:general_rnk}
	Under the assumptions of Proposition \ref{proposition:asy} and the notations of Lemma \ref{lemma:rnk}, and given $\vect{x}_1 \in \Delta^{r_{1}}, \ldots, \vect{x}_k \in \Delta^{r_{k}}, \vect{y}_1 \in \Delta^{r_{k+1}}, \ldots, \vect{y}_h \in \Delta^{r_{k+h}}$,
	it follows that
	\begin{equation}\label{lbound_rank_psi}
	\Rnk{\matr{\Psi}} \geq s - r_{g_k} - \sum_{i=1}^{k} \vert L_i \vert
	\end{equation}
	and
	\begin{equation}\label{lbound_rank_psi_xi}
	\Rnk{\matr{\Psi}+\matr{\Xi}} \geq s - r_{\tilde{g}} - \sum_{i=1}^{k+h} \vert L_i \vert,
	\end{equation}
	where $L_i = \{l \colon x_{il}=0\}$ for $i=1, \ldots, k$, $L_{i+k} = \{l \colon y_{il}=0\}$ for $i=1, \ldots, h$, with $\vert L_i \vert$, $\vert L_{i+k} \vert$ indicating their respective cardinality, and $s$ is defined immediately prior to equation \eqref{h0:tfh}.
\end{corollary}

Taken together, Lemma \ref{lemma:rnk} and Theorem \ref{thm:rnk} serve as example of how to determine the covariance matrix rank upon application of the generalized Wald's framework.
For example, assuming the variables $\{X_{1j}, \ldots, X_{kj}\}$ for $j=1, \ldots, m$ are grouped in $m$ $k$-tuples (thus with $m = n_1 = \ldots = n_k$), each drawn independently from $(X_1, \ldots, X_k)$, a statistic can be built to test the null hypothesis that the random variables $X_1, \ldots, X_k$ are sub-independent \citep{Hamedani2013}, that is 
\begin{equation}\label{h0_sub_independence}
\mathbf{H}_0\colon \psi_{\sum_{i=1}^{k}X_i}(t) = \prod_{i=1}^{k} \psi_{X_i}(t), \quad \quad \text{for all} \quad t \in \mathbb{R},
\end{equation}
where $\psi_{X}$ represents the characteristic function of the random variable $X$.
Sub-independence is a property less stringent than independence, and both these properties imply uncorrelatedness. Moreover the assumption of sub-independence can replace that of independence in several limit theorems \citep{Hamedani2013,Schennach2019}.

\begin{corollary}[Test for sub-independence]\label{cor:sub_independence}
	Under the assumptions of Theorem \ref{thm:rnk} and the null hypothesis of sub-independence \eqref{h0_sub_independence}, such that $ \vect{x}_{1} \ast \ldots \ast \vect{x}_{k} = \vect{z}$, and given $\hat{\vect{z}}_m$ the MLE for the PMV $\vect{z} \in \Delta^{s}_{\text{Int}}$ such that $(\hat{\vect{z}}_{m})_u = \frac{1}{m} \sum_{j=1}^{m} \mathbbm{1}_{\{\sum_{i=1}^{k}X_{ij}=u\}}$ for $u=0, \ldots, s$,
	it holds that
	\begin{equation}\label{s_n}
	\vect{S}_m = \sqrt{m} \left( \hat{\vect{x}}_{1m} \ast \ldots \ast \hat{\vect{x}}_{km} - \hat{\vect{z}} \right) \Asy{m}{\infty} \mathcal{N}(\matr{\Upsilon}),
	\end{equation}
	where $\matr{\Upsilon} = \Smat{\vect{z}} - \matr{\Psi}$ and $\Rnk{\matr{\Upsilon}} = s$.
	
	In particular, given $\hat{\matr{\Upsilon}}_m = \Smat{\hat{\vect{z}}_m} - \hat{\matr{\Psi}}_m$, then
	\begin{equation}\label{asy_chi:sub_independence}
	\vect{S}_m^\prime (\hat{\matr{\Upsilon}}_m^{s})^+ \vect{S}_m \Asy{m}{\infty} \chi^2\left( s \right).
	\end{equation}
\end{corollary}

As the independence of random variables implies their sub-independence, Corollary \ref{cor:sub_independence} also provides a routine to test for complete independence in $k$-way contingency tables \citep{Andersen1974,Bishop2007} (one dimension per random variable $X_i$, $i=1, \ldots, k$), as an alternative to the power divergence family of statistics \citep{CressieRead1984}. In this instance, the null hypothesis of sub-independence, weaker than complete independence, leads to a reduction of free parameters to be estimated, from $\prod_{i=1}^{k} r_i$ (in the model from power divergence statistics) to $\sum_{i=1}^{k} r_i$ (in the framework of the convolution statistics), which is desirable when $k$ is large.
\vskip 3mm

\noindent 4. POWER COMPARISON

We evaluate the performances of the convolution test in terms of type I error and power ($1$ minus type II error), which are the proportion of rejections with significance level $\alpha = 0.05$ under the null and alternative hypothesis respectively, setting Pearson's $\chi^2$ test as the benchmark.
To do so, we simulate the smallest parametrized model that enables the investigation of how samples size, degrees of freedom reduction, and observables distribution affect the convolution test, using different parameters choices. It also allows the transition from the null to alternative hypotheses by modulating a single parameter.

We consider $k=2$, $h=1$ and $X_1, X_2$ are two Bernoulli random variables with parameters $p, q \in (0,1)$ so that $\vect{x}_1 = (1-p, p)$, $\vect{x}_2 = (1-q, q)$ and $\vect{x}_1, \vect{x}_2 \in \Delta^{1}_{\text{Int}}$.
Then we define, for $\rho \in [0,1]$,
\begin{equation}\label{z_rho}
\vect{z}(\rho) = (1-\rho) \vect{x}_1 \ast \vect{x}_2 + \rho (1-a, 0, a),
\end{equation}
where $a = p q + \sqrt{p q (1-p)(1-q)}$ is defined so that $\vect{z}(\rho) = (z(\rho)_0, z(\rho)_1, z(\rho)_2) \in \Delta^{2}_{\text{Int}}$ is the PMV for the distribution of $Z_1 + Z_2$ where $Z_1$ and $Z_2$ are two Bernoulli random variables with parameter $p$ and $q$, respectively, and $\rho$ is their correlation.
The null hypothesis for the goodness-of-fit (GF) test is
$\mathbf{H}_0: X_1 + X_2 \sim \vect{z}(0)$
and we set a family of alternative hypotheses parametrized over $\rho \in (0,1]$ as
$\mathbf{H}_1^\rho: X_1 + X_2 \nsim \vect{z}(\rho)$.
Similarly, the null and the alternative hypotheses of the test for equality in distribution (ED) are defined as
$\mathbf{H}_0\colon X_1 + X_2 \sim Y_1$ with $Y_1 \sim \vect{z}(0)$ and $\mathbf{H}_1^\rho\colon X_1 + X_2 \nsim Y_1$ with $Y_1 \sim \vect{z}(\rho)$, $\rho \in (0,1]$.

To facilitate the following discussion, we set the sample sizes $n_1, n_2$ and $n_3$ for $X_1$, $X_2$ and $Y_1$, respectively, so that $n_1,n_2 \leq n_3$ and $m=\min(n_1, n_2, n_3) = \min(n_1, n_2)$.
Since we are interested in the comparison between the convolution and Pearson's $\chi^2$ statistics, we need to calculate the latter even in the case of unequal sample size, i.e. when $n_1 \neq n_2$.
Thus, we define
\begin{equation*}
P_m^{\text{GF}} = \sum_{j=0}^{2} \frac{(\sum_{i=1}^{m}\mathbbm{1}_{\{X_{1i}+X_{2i}=j\}} - m z(\rho)_j)^2}{m z(\rho)_j} \quad \text{and}
\end{equation*}
\begin{equation*}
\begin{aligned}
P_m^{\text{ED}} =& \sum_{j=0}^{2} \Bigg( \frac{(\frac{n_3}{m+n_3} \sum_{i=1}^{m} \mathbbm{1}_{\{X_{1i}+X_{2i}=j\}} - \frac{m}{m+n_3} \sum_{i=1}^{n_3}\mathbbm{1}_{\{Y_i=j\}} )^2}{\frac{m}{m+n_3}( \sum_{i=1}^{m}\mathbbm{1}_{\{X_{1i}+X_{2i}=j\}}+\sum_{i=1}^{n_3}\mathbbm{1}_{\{Y_i=j\}})} \nonumber \\ 
&+ \frac{(\frac{m}{m+n_3} \sum_{i=1}^{n_3}\mathbbm{1}_{\{Y_i=j\}} - \frac{n_3}{m+n_3} \sum_{i=1}^{m} \mathbbm{1}_{\{X_{1i}+X_{2i}=j\}} )^2}{\frac{n_3}{m+n_3}( \sum_{i=1}^{m}\mathbbm{1}_{\{X_{1i}+X_{2i}=j\}}+\sum_{i=1}^{n_3}\mathbbm{1}_{\{Y_i=j\}})}\Bigg)
\end{aligned}
\end{equation*}
for Pearson's goodness-of-fit and equality in distribution testing statistic, respectively.
In particular, $n_1 + n_2 - 2m$ observations will not be used in the computation of $P_m^{\text{GF}}$ and $P_m^{\text{ED}}$.

We define the convolution statistic with fixed rank $r=1, 2$ from the notation in Proposition \ref{proposition:asy} and Proposition \ref{thm:rkred} as $\vect{V}_m^\prime (\hat{\matr{\Psi}}_m^r)^{+} \vect{V}_m$ and $\vect{W}_m^\prime ((\hat{\matr{\Psi}}_m+\hat{\matr{\Xi}}_m)^r)^{+} \vect{W}_m$.
In the case where the $n_1$ and $n_2$ observations from the random
variables $X_1$ and $X_2$, respectively, are all equal, the sample
covariance matrix is null, i.e. $\hat{\matr{\Psi}}_m = \matr{0}$,
and $(\hat{\matr{\Psi}}_m^r)^{+}$ is not well defined. In this
scenario, Pearson's $P_m^{\text{GF}}$ can still be calculated.
As we aim to compare the power gain over Pearson's procedures, we calculate the convolution statistics for goodness-of-fit test as $\vect{V}_m^\prime (\hat{\matr{\Psi}}_m^r)^{+} \vect{V}_m$, where well defined, otherwise we set it to $P_m^{\text{GF}}$. With the same reasoning for the equality in distribution case, for the following simulations we define the convolution statistic as
\begin{align*}
C_{rm}^{\text{GF}} &= \vect{V}_m^\prime (\hat{\matr{\Psi}}_m^r)^{+} \vect{V}_m (1-\mathbbm{1}_{\{\hat{\matr{\Psi}}_m = \matr{0}\}})+ P_m^{\text{GF}} \mathbbm{1}_{\{\hat{\matr{\Psi}}_m = \matr{0}\}} \quad \text{and}\\
C_{rm}^{\text{ED}} &= \vect{W}_m^\prime ((\hat{\matr{\Psi}}_m+\hat{\matr{\Xi}}_m)^r)^{+} \vect{W}_m (1-\mathbbm{1}_{\{\hat{\matr{\Psi}}_m + \hat{\matr{\Xi}}_m = \matr{0}\}}) + P_m^{\text{ED}} \mathbbm{1}_{\{\hat{\matr{\Psi}}_m + \hat{\matr{\Xi}}_m = \matr{0}\}},
\end{align*}
for goodness-of-fit and equality in distribution tests, respectively. Note that $\lim_{m \to \infty} C_{rm}^{\text{GF}}\allowbreak \sim \chi^2(r)$, since $\hat{\matr{\Psi}}_m= \matr{0}$ if and only if $\hat{\vect{x}}_{1n_1}$, $\hat{\vect{x}}_{2n_2} \in \{(1,0), (0,1)\}$, but, for $j=1, 2$, $\lim_{m \to \infty} \Pro{\hat{\vect{x}}_{jn_j} \in  \{(1,0), (0,1)\}} = 0$.
Analogously, also $\lim_{m \to \infty} C_{rm}^{\text{ED}} \sim \chi^2(r)$ holds true.

Moreover, in order not to confound the comparative analysis, we do
not reduce the limiting $\chi^2$ degrees of freedom in the case
where a positive eigenvalue of $\hat{\matr{\Psi}}_m$, or
$\hat{\matr{\Psi}}_m + \hat{\matr{\Xi}}_m$, is set to $0$ for being
smaller than $10^{-\epsilon}$ (here $\epsilon=15$ is the machine
precision from Python's floating point number in Numpy 1.13.1).

Finally, to assess whether deviations from the limit of the convolution statistic are due to the estimate of the covariance matrix pseudo-inverse, we introduce
\begin{equation*}
Z_{rm}^{\text{GF}} = \vect{V}_m^\prime \matr{\Psi}^{+} \vect{V}_m \quad \text{and} \quad
Z_{rm}^{\text{ED}} = \vect{W}_m^\prime (\matr{\Psi} + \matr{\Xi})^{+} \vect{W}_m
\end{equation*}
for $r=1,2$, which are the convolution statistics calculated with the true covariance matrices.

For all these statistics, we evaluate the proportion of hypothesis rejection for the significance level $\alpha =0.05$ by Monte Carlo approximation over $L=100,\!000$ independent instances of the data.
That is, the proportion of rejections for $U$ (assumed to be one between $P^{\text{GF}}_m$, $C^{\text{GF}}_{rm}$, $Z^{\text{GF}}_{rm}$, $P^{\text{ED}}_m$, $C^{\text{ED}}_{rm}$, $Z^{\text{ED}}_{rm}$) is calculated as
\begin{equation}\label{prop_reg:GF}
\frac{1}{L} \sum_{l=1}^{L} \mathbbm{1}_{\{ S_r(U_l) < \alpha \}},
\end{equation}
given $\{U_1, \ldots U_L\}$ are independent and identically distributed from $U$, and $S_r(t) = \Pro{\chi^2(r) \geq t}$ equal to the survival function for a $\chi^2$ distribution with $r$ degrees of freedom.  

In Fig. \ref{fig:h0_pow}, we report the statistical power under $\mathbf{H}_0$ and a range of alternative hypotheses $\mathbf{H}_1^\rho$. We implement these comparisons for small and large samples with respect to $m=\min(n_1, n_2, n_3)$, and with equal and unequal sizes.

To comply with the rule-of-thumb recommendation for Pearson's $\chi^2$ statistic application \citep{CressieRead1984}, the requirement for the expected frequencies $m (\vect{x}_1 \ast \vect{x}_2)_u \geq 1$ for the categories $u=0,1,2$, must be met when $m$ observations are sampled from the distribution of $X_1 + X_2$.
Thus, we select two cases for the parameters $(p, q)$ so that, under small samples $m=10$, all three constraints from the rule-of-thumb are satisfied when $(p, q) = (0.3, 0.8)$, while only one, i.e. $m(\vect{x}_1 \ast \vect{x}_2)_1 \geq 1$, holds
for $(p, q) = (0.1, 0.9)$.
These parameters were selected to check whether the convolution statistic offers a better alternative over Pearson's $\chi^2$, under cases favorable to $P^{\text{GF}}_{m}$ and $P^{\text{ED}}_{m}$, when $(p, q) = (0.3,0.8)$ and sample sizes are equal, or unfavorable, when $n_1 + n_2 - 2m > 0$ observations are excluded and the rule-of-thumb is violated.

For the small sample cases when $(p, q) = (0.3, 0.8)$, $C^{\text{GF}}_{2m}$ provides better power over $P^{\text{GF}}_m$, and the latter over $C^{\text{GF}}_{1m}$, but $C^{\text{GF}}_{2m}$ shows a proportion of rejections that is above $\alpha$ under $\mathbf{H}_0$ (Fig. \ref{fig:h0_pow}a, top left and middle left panels).
When $(p, q) = (0.1, 0.9)$, $C^{\text{GF}}_{1m}$ and $C^{\text{GF}}_{2m}$ have similar behavior which outperforms $P^{\text{GF}}_m$ (Fig. \ref{fig:h0_pow}a, top right and middle right panels).
Equivalent conclusions are inferred for the equality in distribution testing statistic counterparts (Fig. \ref{fig:h0_pow}b, top and middle panels).
For large samples, under $\mathbf{H}_0$, the proportion of rejections becomes closer to $\alpha$ for $C^{\text{GF}}_{2m}$ (Fig. \ref{fig:h0_pow}a, bottom panels) and it coincides for $C^{\text{ED}}_{2m}$ (Fig. \ref{fig:h0_pow}b, bottom panels); in terms of power, convolution statistics outperform Pearson's $\chi^2$, with the exception of $C^{\text{GF}}_{1m}$ and $C^{\text{ED}}_{1m}$ when $(p, q) = (0.3,0.8)$ (Fig. \ref{fig:h0_pow}a,b, bottom left panels).

\begin{figure}[t]
	\begin{center}
		\includegraphics[width=0.8\textwidth]{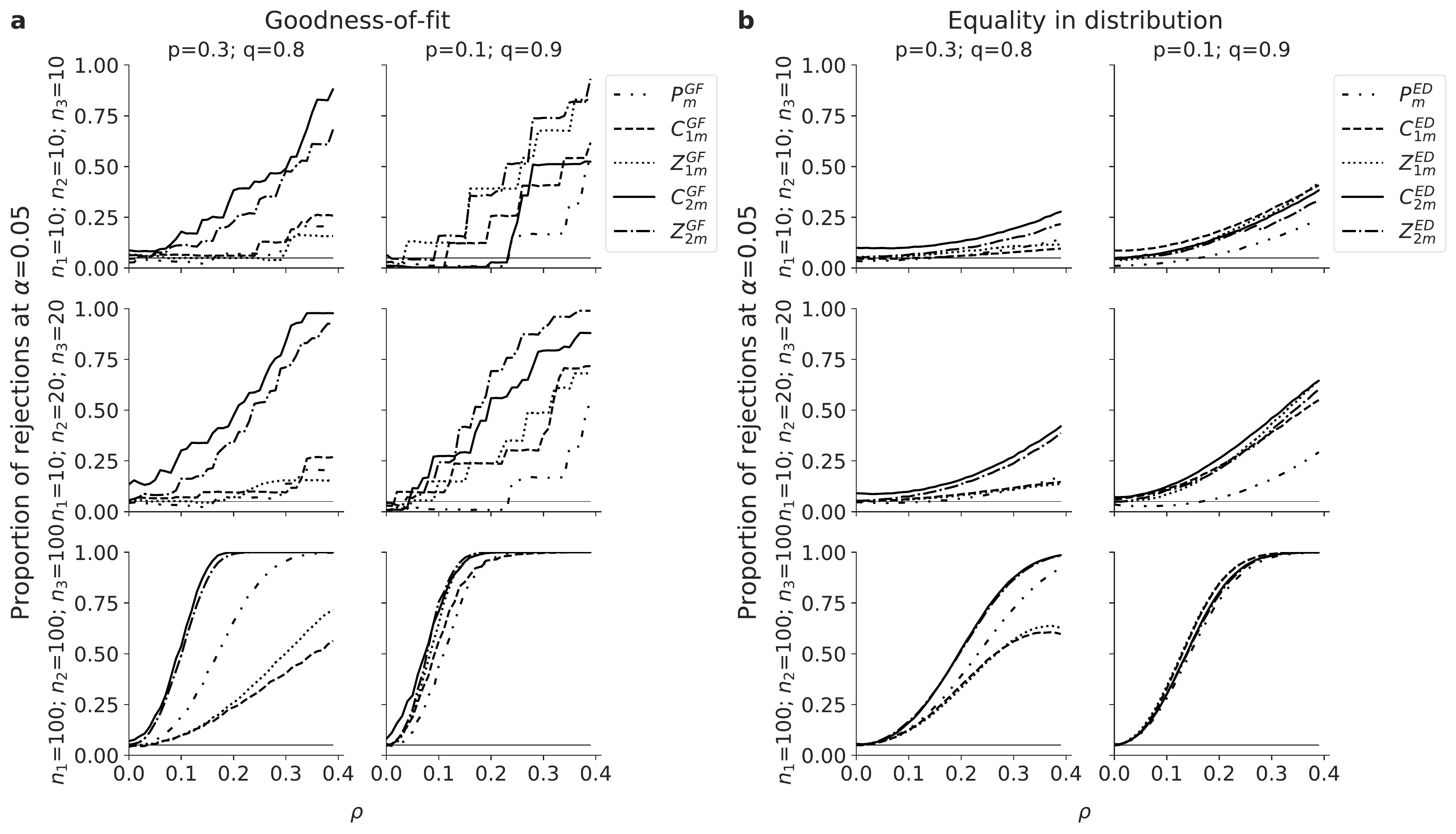}
	\end{center}
	\vspace{-1cm}\caption{
	Power comparison of convolution statistics against those calculated with the true covariance matrix and Pearson's $\chi^2$ statistics, under hypotheses $\mathbf{H}_0$ (for $\rho = 0$) and $\mathbf{H}_1^\rho$ (for $\rho>0$) when testing for goodness-of-fit (a) and equality in distribution (b). A full horizontal line depicts the nominal rejection level $\alpha = 0.05$.
}\label{fig:h0_pow}
\end{figure}

In Fig. \ref{fig:speed} we illustrate the convergence of the rejection rate for $m$ large to the significance level $\alpha$ under $\mathbf{H}_0$ and the power convergence under the alternative hypothesis $\mathbf{H}_1^{0.25}$.
When testing for goodness-of-fit, $C^{\text{GF}}_{2m}$ shows the highest rejection proportion which leads to good power under the alternative hypothesis (Fig. \ref{fig:speed}b), but a slow convergence to $\alpha$ under $\mathbf{H}_0$, reaching peaks of rejection up to $2\alpha$ (Fig. \ref{fig:speed}a).
$C^{\text{GF}}_{1m}$ and $P^{\text{GF}}_m$, instead, have similar behaviors better than $C^{\text{GF}}_{2m}$, with $P^{\text{GF}}_m$ outperforming $C^{\text{GF}}_{1m}$ in its most favorable case ($(p, q) = (0.3,0.8)$, bottom left panels from Fig. \ref{fig:speed}a,b) while the converse holds in the other cases.
When testing for equality in distribution, results for $C^{\text{ED}}_{1m}$ and $P^{\text{ED}}_m$ are similar (Fig. \ref{fig:speed}c,d), while $C^{\text{ED}}_{2m}$ presents a much faster convergence under $\mathbf{H}_0$ than its goodness-of-fit counterpart, as it approaches $\alpha$ already at $m=100$ (Fig. \ref{fig:speed}c).

\begin{figure}[t]
	\begin{center}
		\includegraphics[width=0.7\textwidth]{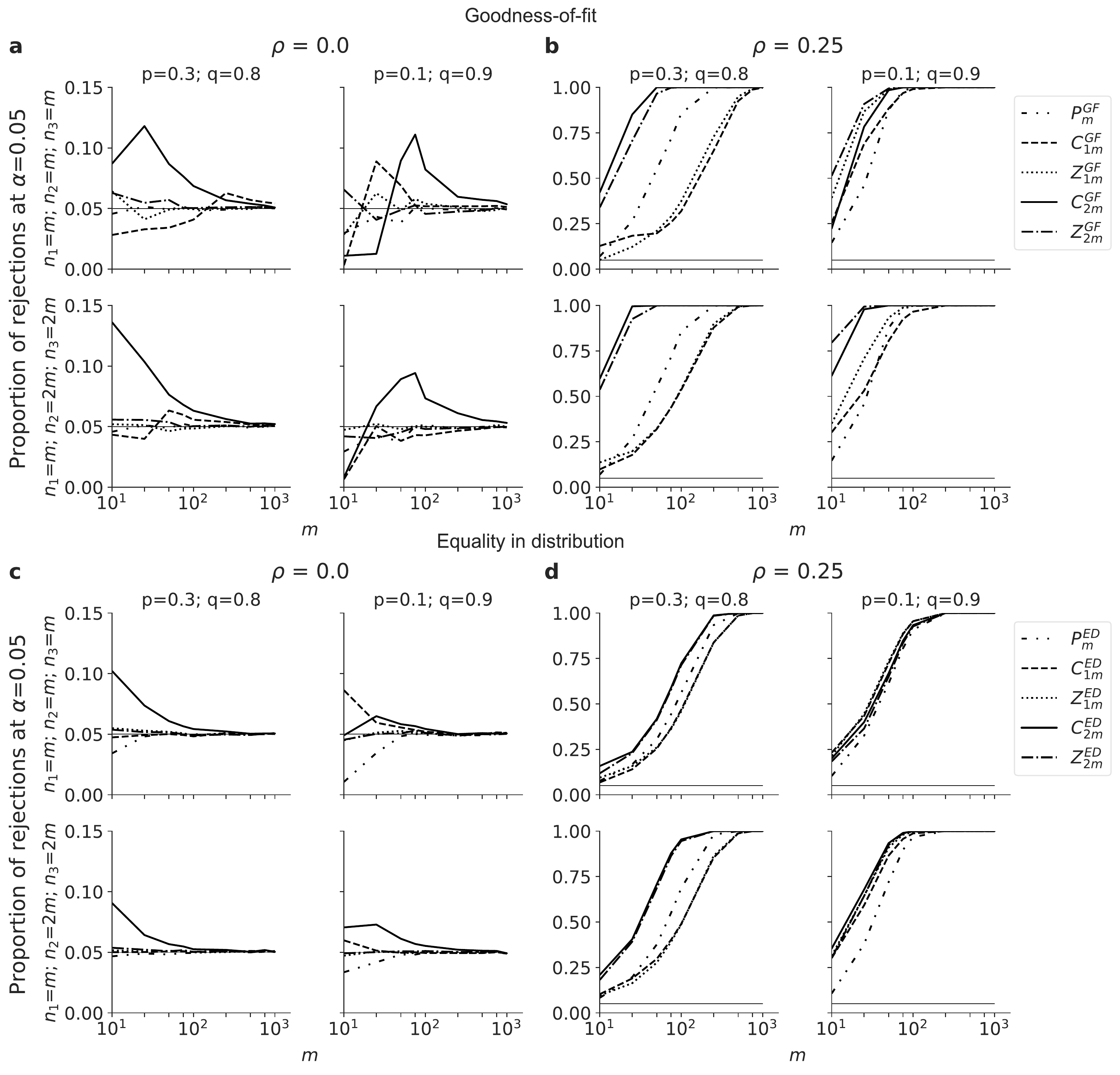}
	\end{center}
	\vspace{-1cm}\caption{
		Comparison of convergence speed of convolution statistics against those calculated with the true covariance matrix and Pearson's $\chi^2$ statistics, as samples size becomes large with $m$. A full horizontal line depicts the nominal rejection level $\alpha = 0.05$.
}\label{fig:speed}
\end{figure}

Together, Figs. \ref{fig:h0_pow} and \ref{fig:speed} suggest a tendency of the convolution statistics to attain a more anti-conservative behavior (type I error higher than $\alpha$), while for Person's $\chi^2$ statistic this is more conservative (type I error lower than $\alpha$).

Lastly, in Fig. \ref{fig:p_pow} we analyze the convolution statistic in the case of covariance matrix that is near a reduced rank form, when its smallest positive eigenvalue approaches zero. Equivalently, this situation occurs if the roots of the PMVs for $X_1$ and $X_2$ are close, i.e. $(p-1)/p - (q-1)/q$ becomes null.
To this end, we fix $q \in (0,1)$, so that $\Rnk{\matr{\Psi}} = 1$ if $p = q$ (but still $\Rnk{\matr{\Psi}+\matr{\Xi}}=2$), and we compare the proportion of rejections when the roots are well separated or close each other, under both $\mathbf{H}_0$ and $\mathbf{H}_1^{0.25}$.

\begin{figure}[t]
	\begin{center}
		\includegraphics[width=0.8\textwidth]{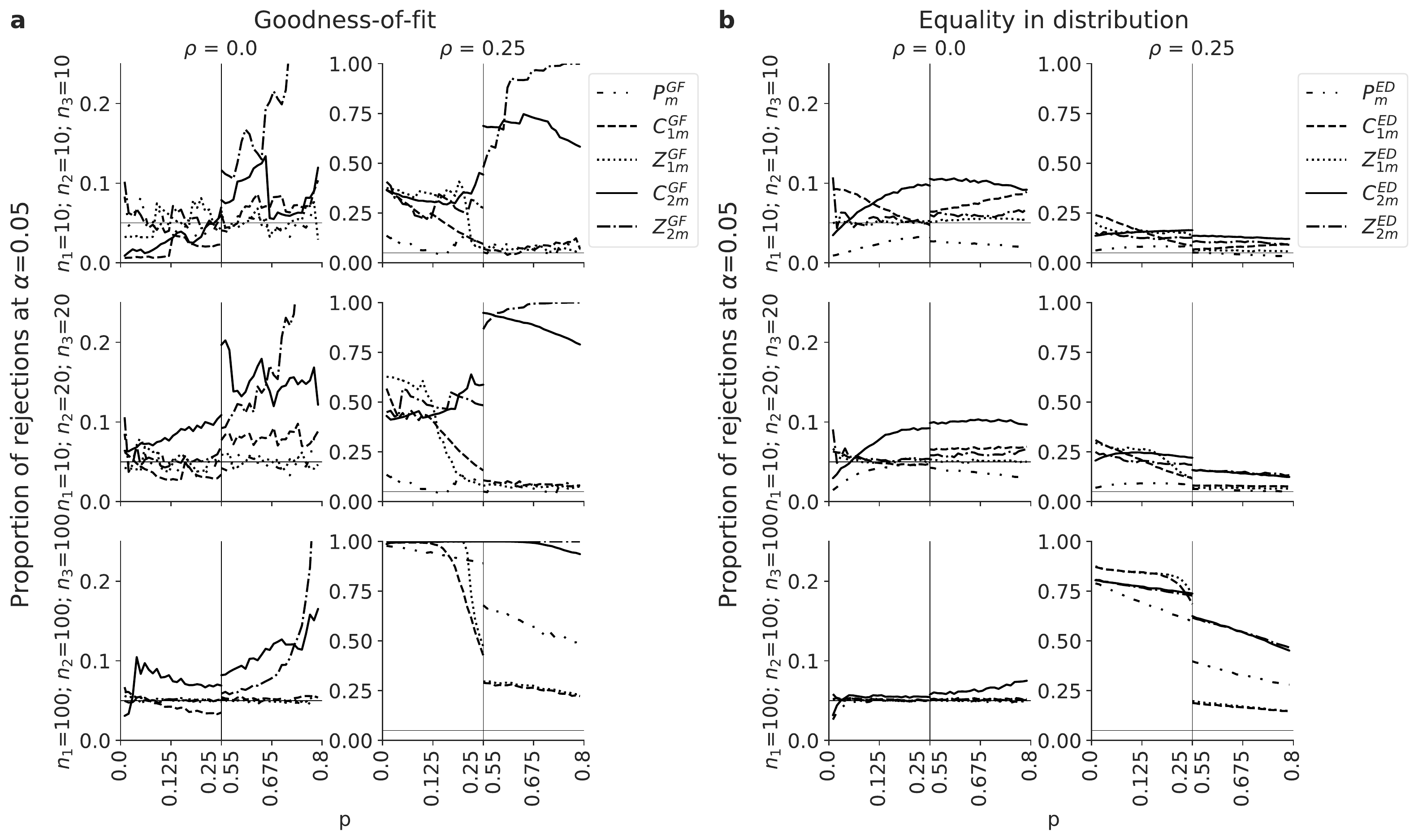}
	\end{center}
	\vspace{-1cm}\caption{
	Proportion of rejections when the smallest positive eigenvalue of $\matr{\Psi}$ tends to zero (when $p$ approaches $q = 0.8$, under the model with $X_1 \sim (1-p, p)$, $X_2 \sim (1-q, q)$) and when the rule-of-thumb for Pearson's $\chi^2$ statistic is violated (when $p$ approaches $0$). Convolution statistics are compared against those calculated with the true covariance matrix and Pearson's $\chi^2$ statistics. A full horizontal line depicts the nominal rejection level $\alpha = 0.05$.
}\label{fig:p_pow}
\end{figure}

For small samples, $C^{\text{GF}}_{1m}$ and $P^{\text{GF}}_m$ present similar performance under $\mathbf{H}_0$ (Fig. \ref{fig:p_pow}a, left panels), while $P^{\text{ED}}_m$ and $C^{\text{ED}}_{1m}$ are, respectively, conservative and anti-conservative (Fig. \ref{fig:p_pow}b, left panels).
The power under $\mathbf{H}_1^{0.25}$ favors the use of $C^{\text{GF}}_{1m}$ over $P^{\text{GF}}_{m}$ for small samples, but, for large samples, $P^{\text{GF}}_{m}$ outperforms $C^{\text{GF}}_{1m}$ when $p$ is near $q$ (Fig. \ref{fig:p_pow}a, right panels). As the spread between $p$ and $q$ widens, eventually $C^{\text{GF}}_{1m}$ performs better than $P^{\text{GF}}_{m}$. This commentary is also true for the equality in distribution statistic counterparts (Fig. \ref{fig:p_pow}b, right panels).

As predicted from Figs. \ref{fig:h0_pow} and \ref{fig:speed} analyses, the behavior of $C^{\text{GF}}_{2m}$ shows higher proportion of rejections.
Furthermore, due to $\Rnk{\matr{\Psi}} = 1$, as $p$ approaches $q$, $C^{\text{GF}}_{2m}$ undertakes a dramatic deviation from the nominal rejection proportion $\alpha$, indicative of the consistency failure of $(\hat{\matr{\Psi}}_m^2)^+$ in the estimation of $\matr{\Psi}^+$ when $p = q$  (Fig. \ref{fig:p_pow}a, bottom left panels).
For large samples, we see that $C^{\text{ED}}_{2m}$ is more powerful than $C^{\text{ED}}_{1m}$, since $\Rnk{\matr{\Psi}+\matr{\Xi}}=2$, and also than $P^{\text{ED}}_m$ when the roots are close  (Fig. \ref{fig:p_pow}b, bottom panels).
The discrepancy between $C^{\text{GF}}_{2m}$ and $C^{\text{ED}}_{2m}$ behaviors highlights the merit for the rank analysis of Section 3 
and shows the danger of setting the same degrees of freedom for the convolution as for Pearson's $\chi^2$ statistics without careful consideration.

To summarize the evidence above, we propose the following as general usage guidelines:
when the number of samples is large or unbalanced, and PMV's roots are distinct, the convolution statistic with maximum rank possible provides the best power;
as PMV's roots become closer or low sample size hiders the estimation of the covariance matrix pseudo-inverse, the convolution statistic with reduced rank should be employed to ensure a good compromise between type I and type II errors control;
Pearson's $\chi^2$ is still recommended over the convolution statistics, but only for small sample data and only if the rule-of-thumb is not violated, in case the type I error must be conservatively controlled and type II error is considered of secondary importance.
Ultimately, it is advisable to run comparative simulations over specific scenarios of interest that largely differ from the one considered above.

\vskip 3mm

\noindent 5. DISCUSSION

In this work, we have shown how to test hypotheses about the sum of discrete random variables using the operation of discrete convolution, from possibly unbalanced datasets and without restrictions to specific distribution classes.
Asymptotic properties of the convolution were combined with the generalized Wald's method to solve the testing problems of: goodness-of-fit, equality in distribution and sub-independence.
These results have been achieved studying the rank of the limiting covariance matrix, from the convolution of PMVs. Of note, such rank was determined as a function of the number of roots shared between the probability generating functions of the random variables considered.
Benchmarking simulations were conducted on a simple parametric model covering several situations of interest and, in most of these, the convolution statistic performed better than Pearson's.

Perspective studies are still required to complete the description of the convolution statistics, including the cases where PMVs' support is not connected, i.e. in $\Delta^{r} \setminus \Delta^{r}_{\text{Int}}$ for $r>0$.
Finally, improved strategies for the estimation of the pseudo-inverse matrix, that are ill-conditioned, should be considered to ameliorate type I and II errors.
\vskip 3mm

\noindent 5. ACKNOWLEDGMENTS

We thank Julia Marchingo, Andrey Kan, Susanne Heinzel and Phil Hodgkin at the Walter and Eliza Hall Institute of Medical Research for collaborating with us on the scientific problem that motivated the development of the convolution statistic and associated tests.
The research leading to these results has received funding from the European Union Seventh Framework Programme (FP7/2007–2013) under grant agreement 317040 (QuanTI) and by Science Foundation Ireland Grant 12IP1263.
On behalf of all authors, the corresponding author states that there is no conflict of interest.

\noindent 6. CODE AVAILABILITY

Python 3 code for simulation and testing is publicly available at
``https://github.com/
GiulioPr/Discrete\_convolution\_statistic''.

\bibliography{Bibliography}

\clearpage

\noindent APPENDIX

\begin{proof}[Proof of Lemma \ref{lemma:equiv}]
	Without loss of generality, it is possible to shift from the lattice $\Lambda(\zeta)$ to the set of integers $\mathbb{Z}$ through the natural isomorphism $\phi \colon \Lambda(\zeta) \rightarrow \mathbb{Z}$, $\phi(\zeta u) = u$ for every $u \in \mathbb{Z}$.
	Using this function, we define $A^\prime_i = a_i \phi(A_i)$ and $B^\prime_j = b_j \phi(B_j)$ for $i=1, \ldots, k$, $j=1, \ldots, h$, so to account for the multiplicative constants $a_1, \ldots, a_k, b_1, \ldots, b_h$ in the variables $A^\prime_1, \ldots, A^\prime_k, B^\prime_1, \ldots, B^\prime_h$ mapping $\Omega$ to $\mathbb{N}\cup\{0\}$, and reduce \eqref{h_0:initial} into 
	\begin{equation}\label{pd:intro_mod1}
	\mathbf{H}_0\colon \phi(a_0) + \sum_{i=1}^{k} A^\prime_i \sim \phi(b_0) + \sum_{i=1}^{h} B^\prime_i.
	\end{equation}
	Given $\tau_i = \min\{j \colon \Pro{A^\prime_i = j}>0\}$ for $i=1, \ldots, k$ and $\tau_{k+i} = \min\{j \colon \Pro{B^\prime_i = j}>0\}$ for $i=1, \ldots, h$, that are well defined since the variables $A^\prime_1, \ldots, A^\prime_k, B^\prime_1, \ldots, B^\prime_h$ are assumed finite, we rewrite \eqref{pd:intro_mod1} as
	\begin{equation}\label{pd:intro_mod2}
	\mathbf{H}_0\colon a_0 + \sum_{i=1}^{k} \tau_i + \sum_{i=1}^{k} (A^\prime_i-\tau_i) \sim b_0 + \sum_{i=1}^{h} \tau_{k+i} + \sum_{i=1}^{h} (B^\prime_i - \tau_{k+i}).
	\end{equation}
	For the null hypothesis \eqref{pd:intro_mod2} to be true, $a_0 + \sum_{i=1}^{k} \tau_i = b_0 + \sum_{i=1}^{h} \tau_{k+i}$ must hold. Otherwise, for example, if $a_0 + \sum_{i=1}^{k} \tau_i < b_0 + \sum_{i=1}^{h} \tau_{k+i}$, by definition of $\tau_1, \ldots, \tau_{k+h}$ we would have
	\begin{equation*}
	\begin{aligned}
	0 &= \Pro{b_0 + \sum_{i=1}^{h} \tau_{k+i} + \sum_{i=1}^{h} (B^\prime_i-\tau_{k+i}) = a_0 + \sum_{i=1}^{k} \tau_i } \\
	&= \Pro{a_0 + \sum_{i=1}^{k} \tau_i + \sum_{i=1}^{k} (A^\prime_i-\tau_i) = a_0 + \sum_{i=1}^{k} \tau_i } \geq \prod_{i=1}^{k}\Pro{A^\prime_i = \tau_i}  > 0,
	\end{aligned}
	\end{equation*}
	that is impossible.
	Therefore \eqref{pd:intro_mod2} is equivalent to
	\begin{equation*}
	\mathbf{H}_0\colon \sum_{i=1}^{k} (A^\prime_i-\tau_i) \sim \sum_{i=1}^{h} (B^\prime_i - \tau_{k+i}),
	\end{equation*}
	which, in turn, can be reduced to the form \eqref{h_0:reduced}
	by defining $X_i = A^\prime_i - \tau_i$ for $i=1, \ldots, k$ and $Y_i = B^\prime_i - \tau_{k+i}$ for $i=1, \ldots, h$ thus accounting for the subtraction the constants in the distribution of $X_i$ and $Y_i$. As a consequence, the support of $X_i$ is $\{0, \ldots, r_i\}$ for some positive integer $r_i \in \mathbb{N}$ and
	\begin{equation}\label{supp_0}
	\Pro{X_i = 0}>0
	\end{equation}
	for every $i = 1, \ldots, k$. The same applies to $Y_1, \ldots, Y_h$.
\end{proof}

\begin{proof}[Proof of Proposition \ref{proposition:mle}]
	Given $a_{ij} \in \{0, \ldots, r_i\}$, the independent sample from $X_{ij}$ for $i=1,\ldots, k$ and $j=1, \ldots, n_i$, the MLE of $\vect{x}_1 \ast \ldots \ast \vect{x}_k$ is the element $\vect{\theta} \in \Delta^{s}$ that maximizes
	\begin{multline*}
	\Pro{X_{11}=a_{11}, \ldots, X_{1n_1}=a_{1n_1}, \ldots, X_{k1}=a_{k1}, \ldots, X_{kn_k}=a_{kn_k} \vert \vect{\theta}} \\
	= \prod_{i=1}^{k}\Pro{X_{i1}=a_{i1}, \ldots, X_{in_i}=a_{in_i}\vert \vect{\theta}}
	\end{multline*}
	On the right hand side, for fixed $i\in \{1, \ldots, k\}$, $\Pro{X_{i1}=a_{i1}, \ldots, X_{in_i}=a_{in_i}\vert \vect{\theta}}$ achieves maximum value for any $\vect{\theta} = \vect{\theta}_1 \ast \ldots \ast \vect{\theta}_k$ such that $\vect{\theta}_i = \hat{\vect{x}}_{in_i}$, with $\vect{\theta}_j \in \Delta^{r_j}$ for any $j$. In particular $\vect{\theta} = \hat{\vect{x}}_{1n_1} \ast \ldots \ast \hat{\vect{x}}_{kn_k}$ maximizes all factors, hence the whole product.
\end{proof}

\begin{proof}[Proof of Lemma \ref{lemma:rnk}]
	We first prove \eqref{ker:psi} by induction on $k$. For $k=2$ we have to show that
	\begin{equation}\label{rnk:k2}
	\Ker{\Tmat{\vect{x}_{(1)}}^\prime} \cap \Ker{\Tmat{\vect{x}_{(2)}}^\prime} =  \Ker{\begin{bmatrix}  \Tmat{\vect{x}_{(1)}}^\prime \\ \Tmat{\vect{x}_{(2)}}^\prime \end{bmatrix}} =  \Ker{\Tmat{\vect{g}_2}^\prime},
	\end{equation}
	where $\vect{g}_2 = \gcd{(\vect{x}_{(1)}, \vect{x}_{(2)})} = \gcd{(\vect{x}_2, \vect{x}_1)} \in \mathbb{R}^{r_{g_2}+1}$ and $r_{g_2}\geq 0$. By definition of $\vect{g}_2$, there exist two coprime vectors $\vect{z}_1 \in \mathbb{R}^{r_1 - r_{g_2}+1}, \vect{z}_2 \in \mathbb{R}^{r_2 - r_{g_2}+1}$ such that $\vect{x}_{(1)} = \vect{z}_1 \ast \vect{g}_2$ and $\vect{x}_{(2)} = \vect{z}_2 \ast \vect{g}_2$ so that, by composition of discrete convolution, we can write
	\begin{equation*}
	\begin{bmatrix}  \Tmat{\vect{x}_{(1)}}^\prime \\ \Tmat{\vect{x}_{(2)}}^\prime \end{bmatrix}= \begin{bmatrix}  \Tmat{\vect{z}_{1}}^\prime \\ \Tmat{\vect{z}_{2}}^\prime \end{bmatrix} \Tmat{\vect{g}_2}^\prime,
	\end{equation*}
	where $\Tmat{\vect{g}_{2}}^\prime \in \mathbb{R}^{r_1+r_2-r_{g_2}+1} \times \mathbb{R}^{r_1+r_2+1}$, $\Tmat{\vect{z}_{1}}^\prime \in \mathbb{R}^{r_2+1} \times \mathbb{R}^{r_1+r_2-r_{g_2}+1}$ and $\Tmat{\vect{z}_{2}}^\prime \in \mathbb{R}^{r_1+1} \times \mathbb{R}^{r_1+r_2-r_{g_2}+1}$.
	As the number of columns of $\begin{bmatrix}  \Tmat{\vect{z}_{1}} & \Tmat{\vect{z}_{2}} \end{bmatrix}^\prime$ is lesser than the number of rows, since $r_1 + r_2 - r_{g_2} +1 \leq r_1 + r_2 +2$  $\Leftrightarrow$ $r_{g_2} +1 \geq 0$,
	to prove \eqref{rnk:k2}, it suffices to show that $\Rnk{\begin{bmatrix} \Tmat{\vect{z}_{1}} & \Tmat{\vect{z}_{2}} \end{bmatrix}^\prime}$ is of full rank $r_1 + r_2 - r_{g_2} +1$. By rank-nullity theorem, this is the case if and only if $\Nul{\begin{bmatrix} \Tmat{\vect{z}_{1}} & \Tmat{\vect{z}_{2}} \end{bmatrix}} = r_{g_2} + 1$.
	The latter is true by coprimeness between $\vect{z}_1$ and $\vect{z}_2$ and by matrix dimensionality, since the solutions $(\vect{a},\vect{b})$ with $\vect{a} \in \mathbb{R}^{r_1 + 1}$, $\vect{b} \in \mathbb{R}^{r_2 +1}$, to the homogeneous system of equations $\Tmat{\vect{z}_1} \vect{a} + \Tmat{\vect{z}_2} \vect{b} = \vect{z}_1 \ast \vect{a} + \vect{z}_2 \ast \vect{b} = \vect{0}$, are characterized by $\vect{a} = \vect{z}_2 \ast \vect{u}$, $\vect{b}=- \vect{z}_1 \ast \vect{u}$ with $\vect{u} \in \mathbb{R}^{r_{g_2} +1}$ vector of free parameters.
	Assuming \eqref{ker:psi} for $k$, we now prove the case $k+1$ to conclude. By inductive step and associative property of convolution, we can write
	\begin{equation}\label{ind_step:Tk}
	\bigcap_{i=1}^{k+1} \Ker{\Tmat{\vect{x}_{(i)}}^\prime} 
	= \Ker{\begin{bmatrix} \Tmat{\vect{g}_k}^\prime \Tmat{\vect{x}_{k+1}}^\prime \\ \Tmat{\vect{x}_{(k+1)}}^\prime \end{bmatrix}} = \Ker{\begin{bmatrix} \Tmat{\vect{g}_k \ast \vect{x}_{k+1}}^\prime \\ \Tmat{\vect{x}_{(k+1)}}^\prime \end{bmatrix}}.
	\end{equation}
	Thus, by the definition of $\vect{g}_{k+1} = \gcd(\vect{x}_{(1)}, \ldots, \vect{x}_{(k+1)}) \in \mathbb{R}^{r_{g_{k+1}} +1}$, the properties of $\gcd$ lead to
	$\vect{g}_{k+1} = \gcd(\vect{g}_{k} \ast \vect{x}_{k+1}, \vect{x}_1 \ast \ldots \ast \vect{x}_{k}) = \vect{g}_{k} \gcd(\vect{x}_{k+1}, \vect{u}_{k+1}) 	$,
	where $\vect{u}_{k+1} \in \mathbb{R}^{\sum_{i=1}^k r_i - r_{g_k} +1}$ is such that $\vect{u}_{k+1} \ast \vect{g}_{k+1} = \vect{x}_1 \ast \ldots \ast \vect{x}_{k}$. In particular, we deduce
	\begin{equation}\label{rk+1:rk}
	r_{g_{k+1}} \geq r_{g_{k}}.
	\end{equation}
	As in the previous step, we introduce $\vect{z}_1 \in \mathbb{R}^{r_{g_k}+r_{k+1}- r_{g_{k+1}}+1}$ and $\vect{z}_2 \in \mathbb{R}^{\sum_{i=1}^{k} r_i- r_{g_{k+1}}+1}$ coprime vectors such that $\vect{z}_1 \ast \vect{g}_{k+1} = \vect{g}_{k} \ast \vect{x}_{k+1}$ and $\vect{z}_2 \ast \vect{g}_{k+1} = \vect{x}_1 \ast \ldots \ast \vect{x}_k$. Thus, resuming \eqref{ind_step:Tk}, 
	\begin{equation*}
	\bigcap_{i=1}^{k+1} \Ker{\Tmat{\vect{x}_{(i)}}^\prime} = \Ker{\begin{bmatrix} \Tmat{\vect{z}_1}^\prime \\ \Tmat{\vect{z}_2}^\prime \end{bmatrix} \Tmat{\vect{g}_{k+1}}^\prime} = \Ker{\Tmat{\vect{g}_{k+1}}^\prime}
	\end{equation*}
	where last equality is analogous as for the case $k=2$, given that the number of columns of $\begin{bmatrix} \Tmat{\vect{z}_1}& \Tmat{\vect{z}_2} \end{bmatrix}^\prime \in \mathbb{R}^{\sum_{i=1}^{k+1} r_i - r_{g_{k}} +2} \times \mathbb{R}^{\sum_{i=1}^{k+1} r_i - r_{g_{k+1}} +1}$ is lesser than the number of rows, that is
	\begin{equation*}
	\sum_{i=1}^{k+1} r_i - r_{g_{k+1}} +1 \leq \sum_{i=1}^{k+1} r_i - r_{g_{k}} +2 \quad \Leftrightarrow \quad r_{g_{k+1}} + 1 \geq r_{g_{k}},
	\end{equation*}
	which holds true by \eqref{rk+1:rk}.
	To prove \eqref{ker:psi_xi}, we note that the first equation therein is established by \eqref{ker:psi}, thus only the second equivalence shall be proved. Once more, we define $\vect{z}_1$ and $\vect{z}_2$ such that such that $\begin{bmatrix} \Tmat{\vect{z}_1} & \Tmat{\vect{z}_2} \end{bmatrix}^\prime \Tmat{\tilde{\vect{g}}}^\prime = \begin{bmatrix} \Tmat{\vect{g}_k} & \Tmat{\bar{\vect{g}}_h} \end{bmatrix} ^\prime$ with $\Tmat{\vect{z}_1}^\prime \in \mathbb{R}^{\sum_{i=1}^k r_i - r_{g_{k}} + 1} \times \mathbb{R}^{\sum_{i=1}^k r_i - r_{\tilde{g}} + 1}, \Tmat{\vect{z}_2}^\prime \in \mathbb{R}^{\sum_{i=1}^k r_i - r_{\bar{g}_{h}} + 1} \times \mathbb{R}^{\sum_{i=1}^k r_i - r_{\tilde{g}} + 1}$. Again, the conclusion is verified by checking that, in the matrix $\begin{bmatrix} \Tmat{\vect{z}_1} & \Tmat{\vect{z}_2} \end{bmatrix} ^\prime$, there are less rows than columns, namely
	\begin{equation*}
	\sum_{i=1}^k r_i +r_{\tilde{g}}+1 \geq r_{g_k} + r_{\bar{g}_h},
	\end{equation*}
	which holds true since $\vect{g}_k \ast \bar{\vect{g}}_h = \gcd(\vect{g}_k, \bar{\vect{g}}_h) \ast \lcm(\vect{g}_k, \bar{\vect{g}}_h)=\tilde{\vect{g}}\ast \lcm(\vect{g}_k, \bar{\vect{g}}_h)$ and $\lcm(\vect{g}_k,\allowbreak \bar{\vect{g}}_h)$ is a divisor of $\vect{z}=\vect{x}_1 \ast \ldots \ast \vect{x}_k = \vect{y}_1 \ast \ldots \ast \vect{y}_h$ (as polynomials), so $\deg\lcm(\vect{g}_k, \bar{\vect{g}}_h) \leq \sum_{i=1}^k r_i $.
\end{proof}

\begin{proof}[Proof of Theorem \ref{thm:rnk}]
	The relations \eqref{thm_kerpsi} and \eqref{thm_kerpsixi} derive as applications of Lemma \ref{lemma:rnk} to \eqref{ob:kerpsi} and \eqref{ob:kerpsixi}, respectively.
	Lastly, \eqref{rank_psi} and \eqref{rank_psi_xi} follow from \eqref{thm_kerpsi} and \eqref{thm_kerpsixi}, respectively, through rank-nullity properties of linear transformations from a finite-dimensional domain.
\end{proof}

\begin{proof}[Proof of Corollary \ref{cor:coprime_rnk}]
	This follows from Theorem \ref{thm:rnk}, as $\Ker{\Tmat{\vect{g}_k}^\prime} = \Ker{\Tmat{\tilde{\vect{g}}}^\prime} = \{(0, \ldots, 0)\} \subseteq \mathbb{R}^{s+1}$ by coprimeness.
\end{proof}

\begin{proof}[Proof of Corollary \ref{cor:general_rnk}]
	It suffices to show \eqref{lbound_rank_psi}, since \eqref{lbound_rank_psi_xi} follows from the same reasoning. Relation \eqref{ker_STi1}, Lemma \ref{lemma:rnk} ans Theorem \ref{thm:rnk} imply
	\begin{equation*}
	\begin{aligned}
	\Ker{\matr{\Psi}} &= \Ker{\Tmat{\vect{g}_{k}}^\prime} \oplus \left( \bigcap_{i=1}^{k}
	\{\vect{v} \in \mathbb{R}^{s+1} \colon \Tmat{\vect{g}_{k}}^\prime \vect{v} \in \Ker{\Smat{\vect{x}_i}} \}
	\right)\\
	& \subseteq \langle \vect{1}_{s} \rangle \oplus \Ker{\Tmat{\vect{g}_{k}}^\prime} \oplus_{i=1}^{k}
	\{\vect{v} \in \mathbb{R}^{s+1} \colon \Tmat{\vect{g}_{k}}^\prime \vect{v} \in E_i \},
	\end{aligned}
	\end{equation*}
	since $\Ker{\Smat{\vect{x}_i}} = \langle \vect{1}_{r_i} \rangle \oplus E_i$, given $E_i = \langle \{\vect{e}^i_l \colon l \in L_i\}\rangle$ with $e^i_{lu} = \delta_{l,u}$ for $u=0, \ldots, r_i$ and $i=1, \ldots, k$. As the dimension of $E_i$ equals the cardinality of $L_i$, \eqref{lbound_rank_psi} follows from
	\begin{equation*}
	\Nul{\matr{\Psi}} \leq 1 + r_{g_k} + \sum_{i=1}^{k} \vert L_i \vert.
	\end{equation*}
\end{proof}

\begin{proof}[Proof of Corollary \ref{cor:sub_independence}]
	To prove \eqref{s_n} we write the relation
	\begin{equation}\label{asy_x_z}
	\sqrt{m} 
	\begin{pmatrix}
	\hat{\vect{x}}_{1m} - \vect{x}_{1} \\
	\vdots  \\
	\hat{\vect{x}}_{km} - \vect{x}_{k} \\
	\hat{\vect{z}}_{m} - \vect{z}  \\
	\end{pmatrix}
	\Asy{m}{\infty}
	\mathcal{N}\left(
	\begin{bmatrix}
	\Smat{\vect{x}_1} & \matr{0} &\dots & \matr{0} & \matr{A}_1\\
	\matr{0} & 	\ddots & \ddots & \vdots & \vdots\\
	\vdots  & \ddots  & \ddots & \matr{0} & \vdots \\
	\matr{0} & \dots & \matr{0} & \Smat{\vect{x}_k} & \matr{A}_k \\
	\matr{A}^\prime_1 & \dots & \dots & \matr{A}^\prime_k & \Smat{\vect{z}}  \\
	\end{bmatrix}
	\right),
	\end{equation}
	where the $\matr{0}$ entries follow from the uncorrelatedness implied by sub-independence, and $\matr{A}_{i} = \vect{D}(\vect{x}_i)( \Tmat{\vect{x}_{(i)}}^\prime - (\vect{z} \dots \vect{z})^\prime )$, given $\vect{D}(\vect{x}_i)$ the matrix with $\vect{x}_i$ at the diagonal and $0$ elsewhere.
	Then, \eqref{s_n} is derived by applying the delta method to \eqref{asy_x_z}, with
	\begin{equation*}
	\matr{\Upsilon} = \matr{\Psi} - \sum_{i=1}^{k} \matr{A}^\prime_i \Tmat{\vect{x}_{(i)}}^\prime
	- \sum_{i=1}^{k} \Tmat{\vect{x}_{(i)}} \matr{A}_i + \Smat{\vect{z}} = \Smat{\vect{z}} - \matr{\Psi},
	\end{equation*}
	since
	\begin{equation*}
	\sum_{i=1}^{k} \Tmat{\vect{x}_{(i)}} \matr{A}_i = \sum_{i=1}^{k} \Tmat{\vect{x}_{(i)}} ( \vect{D}(\vect{x}_i) - \vect{x}_i \vect{x}_i^\prime) \Tmat{\vect{x}_{(i)}}^\prime = 
	\sum_{i=1}^{k} \Tmat{\vect{x}_{(i)}} \Smat{\vect{x}_i} \Tmat{\vect{x}_{(i)}}^\prime = \matr{\Psi}.
	\end{equation*}
	To obtain $\Rnk{\matr{\Upsilon}}=s$,  $\matr{\Upsilon}$ is rewritten as
	\begin{equation*}
	\begin{aligned}
	\matr{\Upsilon} &= \Smat{\vect{z}} - \matr{\Psi} = \\
	&=
	\begin{bmatrix}
	\matr{I} & \Tmat{\vect{x}_{(1)}} & \cdots & \Tmat{\vect{x}_{(k)}}
	\end{bmatrix}
	\begin{bmatrix}
	\Smat{\vect{z}} & \matr{0} &\dots & \dots & \matr{0}\\
	\matr{0} & 	\Smat{\vect{x}_1} &\ddots & \ddots & \vdots\\
	\vdots  & \ddots  & \ddots & \ddots &  \vdots \\
	\vdots & \ddots & \ddots & \ddots & \matr{0} \\
	\matr{0} & \dots & \dots & \matr{0} & \Smat{\vect{x}_k}  \\
	\end{bmatrix}
	\begin{bmatrix}
	\matr{I} \\ -\Tmat{\vect{x}_{(1)}}^\prime \\ \vdots \\ -\Tmat{\vect{x}_{(k)}}^\prime
	\end{bmatrix} \\
	&= \matr{M}_1 \matr{M}_2 \matr{M}_3,
	\end{aligned}
	\end{equation*}
	where $\matr{I}$ is the identity matrix of dimension $s+1 \times s+1$.
	In fact, following the rationale as for Theorem \ref{thm:rnk} and \eqref{ker_STi1}, $\Ker{\matr{M}_2 \matr{M}_3} = \langle \vect{1}_{s} \rangle$, hence
	$\Ker{\matr{\Upsilon}} = \langle \vect{1}_{s} \rangle \oplus \{\vect{v} \in \mathbb{R}^{s+1} \colon \matr{M}_2 \matr{M}_3 \vect{v} \in \Ker{\matr{M}_1} \}$.
	Since $\Ker{\matr{M}_1}$ is a space orthogonal to $\Img{\matr{M}_1^\prime}$, then $\Ker{\matr{M}_1} \subseteq \{(\vect{0},\vect{w}) \in \mathbb{R}^{2s+k+1} \colon \vect{w} \in \mathbb{R}^{s+k}\}$ and therefore
	$\{\vect{v} \in \mathbb{R}^{s+1} \colon \matr{M}_2 \matr{M}_3 \vect{v} \in \Ker{\matr{M}_1}\} = \langle \vect{1}_{s} \rangle$, which leads to $\Rnk{\matr{\Upsilon}} = s$. To conclude, \eqref{asy_chi:sub_independence} derives from Proposition \ref{thm:moore}.
\end{proof}
\end{document}